\documentclass[a4paper,10pt]{article}

\usepackage{amsmath}
\usepackage{german, ngerman}
\usepackage[german]{babel}
\usepackage[latin1]{inputenc}
\usepackage{graphics}
\usepackage{verbatim}
\usepackage{listings}
\usepackage{dsfont}
\usepackage{amsmath}
\usepackage{amsfonts}
\usepackage{amssymb}
\usepackage{graphicx}
\usepackage{geometry}
\usepackage{graphics}
\usepackage{url}
\usepackage{framed}
\usepackage{color}
\usepackage{textcomp}

\newtheorem{set2}{Satz}[section]
\newtheorem{theorem}[set2]{Theorem}

\newtheorem{definition}[set2]{Definition}

\newtheorem{lemma}[set2]{Lemma}
\newtheorem{notation[set2]}{Notation}

\newtheorem{proposition}[set2]{Proposition}
\newtheorem{remark}[set2]{Remark}

\newcommand{\ep}{\hfill{$\blacksquare$}}
\newenvironment{proof}[1][Proof]{\textbf{#1.} }{\
\\}

\def\Xint#1{\mathchoice{\XXint\displaystyle\textstyle{#1}}{\XXint\textstyle\scriptstyle{#1}}
{\XXint\scriptstyle\scriptscriptstyle{#1}}{\XXint\scriptstyle\scriptscriptstyle{#1}}\!\int}
\def\XXint#1#2#3{{\setbox0=\hbox{$#1{#2#3}{\int}$}
\vcenter{\hbox{$#2#3$}}\kern-.5\wd0}}
\def\dashint{\Xint-}

\newcommand{\R}{\mathbb{R}}
\newcommand{\N}{\mathbb{N}}
\newcommand{\C}{\mathcal}
\newcommand{\ol}{\overline}
\newcommand{\dx}{\,\mathrm dx}
\newcommand{\dt}{\,\mathrm dt}
\newcommand{\ds}{\,\mathrm ds}

\newcommand{\esssup}{\mathrm{ess\,sup}}

\allowdisplaybreaks

\begin{document}

\date{December 2, 2010}
\selectlanguage{english}



\setlength{\parindent}{0em}
\setlength{\parskip}{0.5em}



\begin{center}
\Large
	Existence results for diffuse interface models describing phase separation and damage
\end{center}
\begin{center}
	Christian Heinemann\footnote{Weierstrass Institute for Applied Analysis and Stochastics (WIAS), Mohrenstr. 39, 10117 Berlin.\\
	E-mail: \texttt{christian.heinemann@wias-berlin.de} and
\texttt{christiane.kraus@wias-berlin.de}\\
		This project is partially supported by the DFG project A~3
		``Modeling and sharp interface limits of local and non-local 
		generalized Navier-Stokes-Korteweg Systems''.},
	Christiane Kraus$^1$
\end{center}
\begin{center}
	December 2, 2010
\end{center}
\vspace{2mm}

\noindent
{\it Abstract:}
	In this paper, we analytically investigate multi-component Cahn-Hilliard and Allen-Cahn systems which are coupled with elasticity and
	uni-directional damage processes.
	The free energy of the system is of the form $\int_\Omega\frac{1}{2}\mathbf\Gamma\nabla c:\nabla c
	+\frac{1}{2}|\nabla z|^2+W^\mathrm{ch}(c)+W^\mathrm{el}(e,c,z)\,\mathrm dx$
	with a  polynomial or logarithmic chemical energy density  $W^\mathrm{ch}$, an
	inhomogeneous elastic energy density   $W^\mathrm{el}$ and a quadratic structure of the gradient of the damage variable $z$.
	For the corresponding elastic Cahn-Hilliard and Allen-Cahn
	systems coupled with uni-directional damage processes, we present an appropriate
	notion of weak solutions and prove existence results based on certain regularization methods
	and a higher integrability result for the strain $e$.
\\[2mm]

\section{Introduction}
	Phase separation and damage are common phenomena in many fields,
	including material sciences, biology and chemical
	reactions. Such microstructural processes take place to reduce the
	total free energy, which may include bulk chemical energy, interfacial 
	energy and elastic strain energy.

	The knowledge of the mechanisms inducing phase separation and damage 
	processes is very important for technological applications as for
	instance in the area of micro-electronics due to the ongoing 
	miniaturization. The materials used in this area  are  typically
	alloys consisting of mixtures of several components (cf. \cite{HCW91}).

	Phase separation and damage processes are usually described
	by two separate models in the mathematical literature. To describe
	phase separation processes for alloys, phase-field models of
	Cahn-Hilliard and Allen-Cahn type coupled  with elasticity are
	well adapted. On the other hand, damage processes for standard materials are
	often modelled as unilateral processes within a gradient-theory \cite{FN96}.
	A phase-field approach which describes {\it both} phase
	separation and damage processes in a unifying model
	has been recently introduced in \cite{HeiKraus}.
	
	The main objective of this work is to prove under general
        assumptions existence results for  multi-component systems where Cahn-Hilliard as well as Allen-Cahn equations
				are \textit{coupled} with rate-dependent damage differential inclusions for elastic
        materials.
        We are interested in free energies of the system which may contain a 
        chemical energy of {\it logarithmic} or {\it polynomial type}, an {\it inhomogeneous}
        elastic energy and a {\it quadratic term} of the gradient of the damage
        variable. To this end, we establish some regularization methods which
        enable us to show existence results for gradient terms  $|\nabla z|^p$
        of the damage variable $z$ in the free energy even if the assumption $p>n$ ($n$
        space dimension) is dropped.
        In contrast to \cite{Mielke06,HeiKraus} now the physical meaningful term $|\nabla z|^2$
        can be treated, cf.~\cite{FN96, Fre02}. In addition, we also provide a higher
        integrability result for the strain tensor. As a consequence, the chemical free
        energy may also have a logarithmic structure such that we are not restricted
        to polynomial growth as in \cite{HeiKraus}.
	We focus on the modelling of  rate-dependent damage processes but we would like to mention that the
        results of this work can be extended to rate-independent systems 
        (i.e.~the dissipation potential is homogeneous of degree one) 
        by some modifications.
	In the following, we will introduce the model formally. 

	The elastic material we want to consider in this work is an $N$ component alloy occupying a bounded
	Lipschitz domain $\Omega\subseteq\mathbb R^n$.
	To account for phase separation, deformation and damage processes in \textit{one} model,
	a state at a fixed time point is described by a triple $(u,c,z)$ where $u:\Omega\rightarrow\mathbb R^n$ denotes the
	deformation, $c:\Omega\rightarrow\mathbb R^N$ the vector of the chemical concentrations and $z:\Omega\rightarrow\mathbb R$ the damage
	variable.
	The meaning of the variables and its governing evolutionary process is explained more explicitly below.
	
	The mixture of the alloy is described by a phase field vector
	$c=(c_1,\ldots,c_N)$, where the element $c_k$ for $k=1, \dots,N$,
	denotes the concentration of the component $k$.
	Therefore, we will restrict the state space for $c$ to the physically
	meaningful condition $\sum_{j=1}^N c_j=1$ in $\Omega$.
	The constraint $c_k>0$, $k=1,\ldots,N$, in $\Omega$ is also used for logarithmic chemical potentials (see below).

	If an alloy is cooled down below a critical temperature
	then usually  spinodal decomposition and
	coarsening phenomena occur. Well established models for describing
	such effects are the Cahn-Hilliard and Allen-Cahn equations, which
	describe mass preserving and mass non-preserving phase
	separation in solids, cf.~\cite{Cahn61, LC78, Gurtin96,AC79,CN94} for modelling aspects.
	Analytical investigations of Cahn-Hilliard equations can be found in
	\cite{Bonetti02, GarckeHabil, CMP00, Garcke052, Garcke05, Pawlow} and for Allen-Cahn equations in \cite{BS04, CP08, CGPS10}, respectively.
	The essential difference between these two equations is
	that the Cahn-Hilliard equation is a fourth order parabolic evolutionary equation
	expressible as a $H^{-1}$ gradient flow of the free energy with
	respect to $c$ whereas the Allen-Cahn equation is a second order
	parabolic equation arising from an $L^2$ gradient flow.
	More precisely,
	\begin{align}
	\label{eqn:diffusionEquation}
		\begin{aligned}
			&\text{Allen-Cahn:}&&\partial_t c = -\mathbb M \Big(-\mathrm{div}(\mathbf\Gamma\nabla c)+W_{,c}^\mathrm{ch}(c)+W_{,c}^\mathrm{el}(e(u),c,z)\Big),\\
			&\text{Cahn-Hilliard:}&&\partial_t c = \mathrm{div}\Big(\mathbb M\nabla\Big(
				-\mathrm{div}(\mathbf\Gamma\nabla c)+W_{,c}^\mathrm{ch}(c)+W_{,c}^\mathrm{el}(e(u),c,z)\Big)\Big).
		\end{aligned}
	\end{align}
	Here, $W^\mathrm{ch}$ denotes the chemical energy density, $W^\mathrm{el}$ the elastic energy density,
	$\mathbb M$ the mobility matrix satisfying $\sum_{l=1}^N \mathbb M_{kl}=0$ for all $k=1,\ldots,N$
	and $\mathbf\Gamma$ the gradient energy tensor which is a fourth order symmetric and positive definite tensor,
	mapping matrices from $\mathbb R^{N\times n}$ into itself.

	In this work, $W^\mathrm{ch}$ may be a chemical energy density of polynomial type, i.e. $W^\mathrm{ch}(c)=W^\mathrm{ch,pol}(c)$, or
	of logarithmic type, i.e. $W^\mathrm{ch}(c)=W^\mathrm{ch,log}(c)$ (see \eqref{eqn:growthEst8}).
	Note that phase separation only arises if the matrix $A$ in \eqref{eqn:growthEst8} is non-positive definite since the
	first term in \eqref{eqn:growthEst8} is convex.

	Elastic behaviour is modelled by a deformation variable $u$ so that each material point $x\in\Omega$ from the
	reference configuration is located at $x+u(x)$.
	We use the assumption that the strain $e$ is sufficiently small so
	that we can work with the linearized strain tensor given by $e(u)=\frac 12(\nabla u+(\nabla u)^t)$.
	In this work, we will neglect inertia effects
	$\rho\ddot u$ and volume forces $l$. Therefore, the momentum balance equation $\mathrm{div}(\sigma)+l=\rho\ddot u$ from the continuum mechanics
	becomes a quasi-static force equation, i.e.
	\begin{align}
	\label{eqn:quasiStaticEq}
		\mathrm{div}(\sigma)=0.
	\end{align}
	The stress tensor $\sigma$ is defined by $W_{,e}^\mathrm{el}$, i.e. as
	the derivative of the elastic energy with respect to the strain.
	
	Analytical results for multi-component Cahn-Hilliard equations coupled with elastic deformations can be found in
	\cite{GarckeHabil} while Allen-Cahn systems with elasticity are, for instance,
	studied in \cite{Blesgen05}.
	Finite element error estimates of Cahn-Hilliard equations with logarithmic free energies and concentration dependent mobilities
	are derived in \cite{BB99}.
	Recent numerical results for Cahn-Hilliard and Allen-Cahn equations can be found in \cite{BM10}.
	It turns out that different elastic
	moduli of the phases in the mixture influence the rate of coarsening and the morphology
	of the phases decisively \cite{DM00}.
	Numerical investigations of elastic Cahn-Hilliard systems are conducted in \cite{GRW01}.
	
	The damage process we want to consider in this paper is uni-directional,
	i.e.~it can only increase in time and the material is not able to heal itself.
	The phase field variable $z$ satisfying $0\leq z\leq 1$ is interpreted as damage in a
	way that $z(x)=1$ stands for a non-damaged
	and $z(x)=0$ for a maximally damaged material point $x \in \Omega$.
	We assume that the damage in our model is \textit{not complete} which means that a maximal damaged part has still elastic properties.
	These constraints lead to a differential inclusion formulation for the
	evolution of $z$ which relates the derivative of the energy dissipation of the system depending on the rate of damage with the derivative of
	the free energy with respect to $z$.
	More precisely, we consider the doubly nonlinear differential inclusion
	(cf. \cite{FN96})
	\begin{align}
	\label{eqn:differentialInclusion}
			0\in\partial\rho(\partial_t z)-\Delta z+W_{,z}^\mathrm{el}(e(u),z)
				+\partial I_{[0,\infty)}(z).
	\end{align}
	The energy dissipation density due to damage progression is given by $\rho$ where we assume the structure
	\begin{align*}
		\rho(\dot z)=-\alpha \dot z+\frac\beta2|\dot z|^2+ I_{(-\infty,0]}(\dot z)
	\end{align*}
	with $\alpha,\beta>0$.
	Because of the quadratic term $\frac\beta2|\dot z|^2$ the damage evolution is called rate-dependent whereas $\beta=0$ would correspond to
	rate-independent systems.
	See \cite{Mielke05,Efendiev06, Mielke10, MT09} for analytical results on rate-independent damage
	and numerical experiments (without phase separation). We also refer to \cite{BSS05, FK09} for further
	analytical investigations of damage models.
	In comparison to \cite{HeiKraus} we use a gradient-of-damage theory with the Laplacian $\Delta z$ in \eqref{eqn:differentialInclusion}
	instead of a $p$-Laplacian $\mathrm{div}(|\nabla z|^{p-2}\nabla z)$ with $p>n$.

	In conclusion, the systems we would like to consider in this work are governed by \eqref{eqn:diffusionEquation}, \eqref{eqn:quasiStaticEq}
	and \eqref{eqn:differentialInclusion} and can be rewritten as
	\begin{align}
	\begin{aligned}
		\left.
		\begin{cases}
			\partial_t c = -\mathcal S w&\text{ in }\Omega_T,\\
			w=\mathbb P(-\mathrm{div}(\mathbf\Gamma\nabla c)+W_{,c}^\mathrm{ch}(c)+W_{,c}^\mathrm{el}(e(u),c,z))&\text{ in }\Omega_T,\\
			\mathrm{div}(\sigma)=0&\text{ in }\Omega_T,\\
			\partial\rho(\partial_t z)-\Delta z+W_{,z}^\mathrm{el}(e(u),z)+\partial I_{[0,\infty)}(z)\ni 0&\text{ in }\Omega_T,\\
		\end{cases}
		\right\}
	\end{aligned}
	\label{eqn:unifyingModelClassical}
	\tag{{$S_0$}}
	\end{align}
	where $w$ denotes the chemical potential.
	Here, the matrix $\mathbb P$  denotes the orthogonal
	projection of $\mathbb R^N$ onto the tangent space
	$T\Sigma=\big\{x\in\mathbb R^N\,|\,\sum_{k=1}^N x_k=0\big\}$ of the
	affine plane\linebreak $\Sigma:=\big\{x\in\mathbb R^N\,\big|\,\sum_{l=1}^N x_k=1\big\}$.
	The operator $\mathcal S$ determines whether we have an Allen-Cahn or a Cahn-Hilliard type diffusion of the system.
	More precisely,
	\begin{align}
	\begin{aligned}
		&\text{Allen-Cahn:}&&\mathcal S:L^2(\Omega;\mathbb R^N)\rightarrow L^2(\Omega;\mathbb R^N),&&\C S(f):= \mathbb Mf,\\
		&\text{Cahn-Hilliard:}&&\mathcal S:H^1(\Omega;\mathbb R^N)\rightarrow \big(H^1(\Omega;\mathbb R^N)\big)^\star,&&
			\C S(f):=\left\langle \mathbb M\nabla f,\nabla\cdot\right\rangle_{L^2}.\\
	\end{aligned}
	\label{eqn:operatorS}
	\end{align}
	In the Cahn-Hilliard case, the operator $\mathcal S$ is invertible when restricted to $\mathcal S:Y\rightarrow\mathcal D$
	where the spaces $Y$ and $\mathcal D$ are defined as
	\begin{align}
		&Y:=\Big\{c\in H^1(\Omega;\mathbb R^N)\,|\,\int_\Omega c=0,\;\sum_{k=1}^N c_k=0\Big\},\notag\\
		&\mathcal D:=\Big\{c^\star\in \left(H^1(\Omega;\mathbb R^N)\right)^\star\,|\,\langle c^\star,c\rangle_{(H^1)^\star\times H^1}=0
		\text{ for all }c=d(x)(1,\ldots, 1),\notag\\
	\label{eqn:Dspace}
		&\hspace{6em}\text{ where }d\in H^1(\Omega)\text{ and for all }c=e_k,\;k=1,\ldots N\Big\}\;\;\;\;e_k:\;k\text{-th unit function}.
	\end{align}
	We need to impose some restrictions on the mobility matrix $\mathbb
	M$.  We assume that  $\mathbb M$ is symmetric and positive definite on
	the tangent space $T\Sigma$. In addition, due to the constraint
	$\sum_{k=1}^N c_k=1$, $\mathbb M$ has to satisfy the property
	$\sum_{l=1}^N \mathbb M_{kl}=0$ for all $k=1,\ldots,N$. Note, that
	$\mathbb M= \mathbb M \, \mathbb P$. 
 
	We abbreviate  $ D_T:=(0,T)\times D$ and $ (\partial\Omega)_T:=(0,T)\times
	\partial\Omega$, where $D\subseteq\partial\Omega$ with $\mathcal
	H^{n-1}(D)>0$ denotes the Dirichlet boundary. The initial-boundary conditions of our systems are summarized as follows:
	\begin{align}
	\begin{aligned}
			c(0)=c^0&\text{ in }\Omega,&\qquad\qquad\sigma\cdot\overrightarrow{\nu}=0&\text{ on }(\partial\Omega)_T \setminus D_T,\\
			z(0)=z^0&\text{ in }\Omega,&\qquad\qquad\mathbf\Gamma\nabla c\cdot\overrightarrow{\nu}=0&\text{ on }(\partial\Omega)_T,\\
			u=b&\text{ on } D_T,&\qquad\qquad\nabla z\cdot \overrightarrow{\nu}=0&\text{ on }(\partial\Omega)_T
	\end{aligned}
	\label{eqn:unifyingModelIBC}
	\tag{IBC}
	\end{align}
	and additionally for Cahn-Hilliard systems
	\begin{align}
	\begin{aligned}
		&\mathbb M\nabla w\cdot\overrightarrow{\nu}=0&&\text{ on }(\partial\Omega)_T,
	\end{aligned}
	\tag{IBC}
	\end{align}
	where  $\overrightarrow{\nu}$ is
	the unit normal on $\partial\Omega$ pointing outward and
	$b$ the boundary value function on the Dirichlet boundary $D$. The initial values are subject to $0\leq z_0\leq 1$ and
	$c^0\in \Sigma\cap\mathbb R_{>0}^N$ a.e. in $\Omega$.
	In the following, we assume that $b$ can be suitably extended to a function on $\ol{\Omega_T}$.\\

	The paper is organized as follows: In Chapter 2, we introduce an appropriate notion of weak solutions for the system \eqref{eqn:unifyingModelClassical}.
	To handle the differential inclusion rigorously, we adapt the concept of energetic solutions originally introduced in the context of
	rate-independent systems (see for instance \cite{Mielke05}) to phase separation systems coupled with rate-dependent damage.
	This approach was firstly presented in \cite{HeiKraus}.
	The main result and their assumptions are stated at the end of Chapter 2.
	
	In Chapter 3, we prove existence of weak solutions for a regularization of system \eqref{eqn:unifyingModelClassical}
	expressed in classical formulation as
	\begin{align}
		\left.
		\begin{cases}
			\partial_t c = -\mathcal S w&\text{ in }\Omega_T,\\
			w=\mathbb P(-\mathrm{div}(\mathbf\Gamma\nabla c)+W_{,c}^{\mathrm{ch,pol}}(c)+W_{,c}^\mathrm{el}(e(u),c,z)
				+\varepsilon\partial_t c)&\text{ in }\Omega_T,\\
			\mathrm{div}(\sigma)+\varepsilon\mathrm{div}(|\nabla u|^2\nabla u)=0&\text{ in }\Omega_T,\\
			\partial\rho(\partial_t z)-\Delta z-\varepsilon\mathrm{div}(|\nabla z|^{p-2}\nabla z)+W_{,z}^\mathrm{el}(e(u),z)
				+\partial I_{[0,\infty)}(z)\ni 0&\text{ in }\Omega_T,\\
		\end{cases}
		\right\}
	\label{eqn:regUnifyingModelClassical}
	\tag{{$S_\varepsilon$}}
	\end{align}
	where $W^{\mathrm{ch,pol}}$ and $W^{\mathrm{el}}$ satisfy certain polynomial growth conditions and $p>n$.
	The initial-boundary conditions are
	\begin{align}
		\text{\eqref{eqn:unifyingModelIBC} with }(\sigma+\varepsilon|\nabla u|^2\nabla u)\cdot\overrightarrow{\nu}=0\text{ instead of }\sigma\cdot\overrightarrow{\nu}=0.
	\label{eqn:regUnifyingModelIBC}
	\tag{IBC$_\varepsilon$}
	\end{align}
	It turns out that the weak solutions of the regularized system have the following regularities:\\
	$c\in H^1(0,T;L^2(\Omega;\mathbb R^N))$,
	$\nabla u\in L^4(\Omega_T;\mathbb R^{n\times n})$ and $\nabla z\in L^p(\Omega_T;\mathbb R^n)$
	(with $p>n$ as above). They are constructed by adapting the approximation techniques developed in \cite{HeiKraus}.
	
	The limit problem $\varepsilon\searrow 0$ for \eqref{eqn:regUnifyingModelClassical} corresponding to \eqref{eqn:unifyingModelClassical}
	with $W^\mathrm{ch}=W^\mathrm{ch,pol}$ is solved in Chapter 4.
	The displacement field $u$ obtained in this process has
        $H^1(\Omega;\mathbb R^n)$-regularity in the first instance. To establish
        existence results for chemical free energies of logarithmic type, we
        prove a higher integrability result for $\nabla u$ in Chapter 5, which
        is based on some ideas of \cite{EL91, GarckeHabil, Garcke05}. 

	Finally, Chapter 6 is devoted to logarithmic free energies for the
        concentration $c$. 
        Following the approach in \cite{GarckeHabil, Garcke05}, we use a suitable regularization $W^{\mathrm{ch},\delta}$ with polynomial
	growth of the logarithmic free energy
	density $W^\mathrm{ch,log}$ to obtain a solution for \eqref{eqn:unifyingModelClassical}.
        Using this regularization, the chemical components $c_k$ become strictly positive in the limit.

\vspace{2mm}
	The notation, we will use throughout this paper, is collected in the following.\\\\
	\textit{Spaces and sets.}\vspace*{0.4em}\\
	\begin{tabular}{p{6em}l}
			$W^{1,r}(\Omega;\mathbb R^n)$ & standard Sobolev space\\
	\end{tabular}\\
	\begin{tabular}{p{6em}l}
			$W^{1,r}_+(\Omega)$ &functions of $W^{1,r}(\Omega)$ which are non-negative almost everywhere\vspace{0.2em}\\
	\end{tabular}\\
	\begin{tabular}{p{6em}l}
			$W^{1,r}_-(\Omega)$ &functions of $W^{1,r}(\Omega)$ which are non-positive almost everywhere\vspace{0.2em}\\
	\end{tabular}\\
	\begin{tabular}{p{6em}l}
			$W^{1,r}_D(\Omega;\mathbb R^n)$ &functions of $W^{1,r}(\Omega;\mathbb R^n)$ which vanish
				on $D\subseteq\partial\Omega$ in the sense of traces\vspace{0.2em}\\
	\end{tabular}\\
	\begin{tabular}{p{6em}l}
			$B_R(A)$ &open neighbourhood of $A\subseteq\mathbb R^n$ with thickness $R$\vspace{0.2em}\\
	\end{tabular}\\
	\begin{tabular}{p{6em}l}
			$Q_R(x_0)$ & open cube $\{x\in\mathbb R^n\,|\, \|x-x_0\|_{\infty}<R\}$\vspace{0.2em}\\
	\end{tabular}\\
	\begin{tabular}{p{6em}l}
			$\{f=0\}$ &zero set $\{x\in\ol\Omega\,|\,f(x)=0\text{ a.e.}\}$ of a function $f\in L^1(\Omega)$ defined up to a\\
			&set of measure $0$ and defined uniquely if $f\in W^{1,p}(\Omega)$ for $p>n$ as\\
			&$W^{1,p}(\Omega)\hookrightarrow\C C^0(\ol\Omega)$\vspace{0.2em}\\
	\end{tabular}\\
	\begin{tabular}{p{6em}l}
			$\Omega_T$ &the set $(0,T)\times\Omega$\vspace{0.2em}\\
	\end{tabular}\\
	\vspace*{0.5em}\\
	\textit{Functions, operations and measures.}\vspace*{0.4em}\\
	\begin{tabular}{p{6em}l}
			$[f]^+$ &non-negative part of $f$, i.e. $\max\{0,f\}$\vspace{0.2em}\\
	\end{tabular}\\
	\begin{tabular}{p{6em}l}
			$I_M$ & indicator function of a subset $M\subseteq X$\vspace{0.2em}\\
	\end{tabular}\\
	\begin{tabular}{p{6em}l}
			$\chi_M$ & characteristic function of a subset $M\subseteq X$\vspace{0.2em}\\
	\end{tabular}\\
	\begin{tabular}{p{6em}l}
			$W_{,e}$ &classical derivative of a function $W$ with respect to the variable $e$\vspace{0.2em}\\
	\end{tabular}\\
	\begin{tabular}{p{6em}l}
			$\langle g^\star,f\rangle$ &dual pairing of $g^\star\in (W^{1,r}(\Omega;\R^n))^\star$ and $f\in W^{1,r}(\Omega;\R^n)$\vspace{0.2em}\\
	\end{tabular}\\
	\begin{tabular}{p{6em}l}
			$\partial^\mathrm{Cl} E$ & generalized Clarke's subdifferential of $E$\vspace{0.2em}\\
	\end{tabular}\\
	\begin{tabular}{p{6em}l}
			$\mathrm d E$ &G\^{a}teaux differential of $E$\vspace{0.2em}\\
	\end{tabular}\\
	\begin{tabular}{p{6em}l}
			$p^\star$ &Sobolev critical exponent $\frac{np}{n-p}$ for $n>p$\vspace{0.2em}\\
	\end{tabular}\\
	\begin{tabular}{p{6em}l}
			$\mathrm{diam}(Q)$ &diameter of a subset $Q\subseteq\R^n$\vspace{0.2em}\\
	\end{tabular}\\
	\begin{tabular}{p{6em}l}
			$\mathcal H^n$ &Hausdorff measure of dimension $n$\vspace{0.2em}\\
	\end{tabular}\\
	\begin{tabular}{p{6em}l}
			$\mathcal L^n$ &Lebesgue measure of dimension $n$
	\end{tabular}\vspace*{1em}\\

\section{Existence theorem}
\subsection{Weak formulation}
	The weak notion, we will derive in this section for the
	doubly nonlinear differential inclusion occurring in \eqref{eqn:unifyingModelClassical}, is inspired by the concept
	of energetic solutions
	for rate-independent systems (see for instance \cite{Mielke05}).
	In the rate-independent setting, the differential inclusion is formulated by a global stability condition
	and an energy inequality.
	In \cite{HeiKraus}, we have introduced an approach which uses an energy inequality and a variational inequality
	to handle the rate-dependence coming from the viscosity term $\frac\beta2|\dot z|^2$ in the damage dissipation
	density function $\rho$.
	
	The corresponding G\^ateaux-differentiable free energy $\tilde{\mathcal E}:H^1(\Omega;\mathbb R^n)\times H^1(\Omega;\mathbb R^N)\times (H^1(\Omega)
	\cap L^\infty(\Omega))\rightarrow\mathbb R$ and dissipation
	functional $\tilde{\mathcal R}:L^2(\Omega)\rightarrow\mathbb R$ to system \eqref{eqn:unifyingModelClassical} are given by 
	\begin{equation*}
		\begin{split}
			\tilde{\mathcal E}(u,c,z)&:=\int_\Omega\frac{1}{2}\mathbf\Gamma\nabla c:\nabla c
				+\frac{1}{2}|\nabla z|^2+W^\mathrm{ch}(c)+W^\mathrm{el}(e,c,z)\,\mathrm dx,\\
			\tilde{\mathcal R}(\dot z)&:=\int_\Omega-\alpha \dot z+\frac\beta2|\dot z|^2\,\mathrm dx,
		\end{split}
	\end{equation*}
	with viscosity constants $\alpha,\beta>0$.
	To account for the constraints of $z$, we extend the functionals $\tilde{\mathcal E}$ and $\tilde{\mathcal R}$ above for analytical reasons by indicator functions:
	\begin{align*}
	\begin{aligned}
			&\mathcal E(u,c,z):=\tilde{\mathcal E}(u,c,z)+\int_\Omega I_{[0,\infty)}(z)\,\mathrm dx,
			&&\qquad\mathcal R(\dot z):=\tilde{\mathcal R}(\dot z)+\int_\Omega I_{(-\infty,0]}(\dot z)\,\mathrm dx.
	\end{aligned}
	\end{align*}
	If we equip the space $H^1(\Omega)\cap L^\infty(\Omega)$ with the norm $\|\cdot\|_{H^1\cap L^\infty}
	:=\|\cdot\|_{H^1}+\|\cdot\|_{L^\infty}$ the generalized subdifferential
	$\partial_z^\mathrm{Cl}\mathcal E$ at
	a point $(u,c,z)\in H^1(\Omega;\mathbb R^n)\times H^1(\Omega;\mathbb R^N)\times (H^1(\Omega)\cap L^\infty(\Omega))$ is
	\begin{align}
	\label{eqn:subdifferentialE}
		\partial_z^\mathrm{Cl}\mathcal E(u,c,z)=\left\{\mathrm d_z\tilde{\mathcal E}(u,c,z)+r\in (H^1(\Omega)
			\cap L^\infty(\Omega))^\star\,\Big|\,r\in \partial I_{H_+^1(\Omega)\cap L^\infty(\Omega)}(z) \right\}.
	\end{align}
	The inclusion $L^1(\Omega)\subset (H^1(\Omega)\cap L^\infty(\Omega))^\star$ will be later used for the
	construction of a specific subgradient.
	Using property \eqref{eqn:subdifferentialE}, the differential inclusion in \eqref{eqn:unifyingModelClassical} can be rewritten in a
	weaker form as
	\begin{align*}
		0\in\partial_z^\mathrm{Cl}\mathcal E(u(t),c(t),z(t))+\partial_{\dot z}\mathcal R(\dot z(t)).
	\end{align*}
	
	
	The analytical basis for a formulation of a weak solution is the
        following proposition (a proof of a related result can be found in \cite{HeiKraus}):
	\begin{proposition}
	\label{prop:energeticFormulation}
		Let $(u,c,w,z)\in \C C^2(\ol{\Omega_T};\mathbb R^n\times\mathbb R^N\times\mathbb R^N\times\mathbb R)$ be a smooth solution satisfying
		\eqref{eqn:diffusionEquation} and \eqref{eqn:quasiStaticEq}
		with the initial-boundary conditions \eqref{eqn:unifyingModelIBC}. Then the following two conditions are equivalent:
		\begin{itemize}
			\item[(i)]
					$0 \in \partial^\mathrm{Cl}_z\mathcal E(u(t),c(t),z(t))+\partial_{\dot z}\mathcal R(\dot z(t))$ for all $t\in[0,T]$,
			\item[(ii)]
				the energy inequality
				\begin{align*}
						&\mathcal E(u(t),c(t),z(t))
							+\int_{0}^{t}\langle\mathrm d_{\dot z}\tilde{\mathcal R}(\partial_t z),\partial_t z\rangle\,\mathrm ds
							+\int_0^t\langle\mathcal S w(s),w(s)\rangle\mathrm ds
							\notag\\
						&\qquad\qquad\leq
							\mathcal E(u(0),c(0),z(0))+\int_{\Omega_t}W_{,e}^\mathrm{el}(e(u),c,z):e(\partial_t b)\,\mathrm dx\mathrm ds
				\end{align*}
				for all $0\leq t\leq T$
				and the variational inequality
				\begin{align*}
					0\leq \left\langle \mathrm d_z\tilde{\mathcal E}(u(t),c(t),z(t))+r(t)+\mathrm d_{\dot z}
					\tilde{\mathcal R}(\partial_t z(t)),\zeta\right\rangle
				\end{align*}
				for all $\zeta \in H_-^1(\Omega)\cap L^\infty(\Omega)$ and $r(t)\in \partial I_{H_+^1(\Omega)\cap L^\infty(\Omega)}(z(t))$
				and for all $0\leq t\leq T$.
		\end{itemize}
		If one of the two conditions holds then the following energy
                balance equation is satisfied:
		\begin{align*}
				&\mathcal E(u(t),c(t),z(t))
					+\int_{0}^{t}\langle\mathrm d_{\dot z}\tilde{\mathcal R}(\partial_t z),\partial_t z\rangle\,\mathrm ds
					+\int_0^t\langle\mathcal S w(s),w(s)\rangle\mathrm ds
					\notag\\
				&\qquad\qquad=
					\mathcal E(u(0),c(0),z(0))+\int_{\Omega_t}W_{,e}^\mathrm{el}(e(u),c,z):e(\partial_t b)\,\mathrm dx\mathrm ds.
		\end{align*}
	\end{proposition}
	\textbf{Remarks for Proposition \ref{prop:energeticFormulation}.}
	In contrast to \cite{HeiKraus}, the energy inequality in (ii) compares the energy at the beginning $s=0$ with the energy at an arbitrary
	time $s=t$ instead of $s=t_1$ with $s=t_2$ for $0\leq t_1<t_2\leq T$.
	
	Applying the chain rule on the right hand side of
	$$\C E(u(t),c(t),z(t))-\C E(u(0),c(0),z(0))=\int_0^t\frac{\mathrm d}{\mathrm dt}\tilde{\C E}(u(s),c(s),z(s))\ds$$
	and using
	\eqref{eqn:diffusionEquation}, \eqref{eqn:quasiStaticEq} as well as the variational inequality in (ii), the
	``$\geq$''-part of the energy balance can be shown.


	We will see that in our approach the mathematical analysis of \eqref{eqn:unifyingModelClassical}
	requires several $\varepsilon$-regularization terms (see \eqref{eqn:regUnifyingModelClassical}) to establish the energy and
	variational inequality for the differential
	inclusion and to handle the logarithmic free energy.
	A transition to $\varepsilon\searrow 0$ will finally give us a solution of the limit problem
	\eqref{eqn:unifyingModelClassical}.
	
	Proposition \ref{prop:energeticFormulation} can also be formulated for the regularized system
	\eqref{eqn:regUnifyingModelClassical}
	with the regularized energy
	\begin{align*}
		&\tilde{\mathcal E}_\varepsilon(u,c,z):=
			\int_\Omega\frac{1}{2}\mathbf\Gamma\nabla c:\nabla c
			+\frac{1}{2}|\nabla z|^2+W^{\mathrm{ch,pol}}(c)+W^\mathrm{el}(e,c,z)
			+\frac \varepsilon4|\nabla u|^4+\frac \varepsilon p|\nabla z|^p\,\mathrm dx,\\
		&\mathcal E_\varepsilon(u,c,z):=\tilde{\mathcal E}_\varepsilon(u,c,z)+\int_\Omega I_{[0,\infty)}(z)\,\mathrm dx,
	\end{align*}
	and the initial-boundary conditions \eqref{eqn:regUnifyingModelIBC}.
	Notice that $\mathbb P\partial_t c=\partial_t c$ because of $\partial_t c(t,x)\in T\Sigma$.
	
	We can now give a weak notion of \eqref{eqn:regUnifyingModelClassical} and \eqref{eqn:unifyingModelClassical}.
	(The energy densities $W^{\mathrm{ch,pol}}$ and $W^{\mathrm{el}}$ will satisfy some polynomial growth
	conditions which are specified in the next subsection.)


	\begin{definition}[Weak solution for the regularized system \eqref{eqn:regUnifyingModelClassical}]
	\label{def:weakSolutionRegularized}
		We call a quadruple $q=(u,c,w,z)$ a weak solution of the regularized system
		\eqref{eqn:regUnifyingModelClassical} with the initial-boundary conditions
		\eqref{eqn:regUnifyingModelIBC} if the following properties are satisfied:
		\begin{enumerate}
			\renewcommand{\labelenumi}{(\roman{enumi})}
			\item
				the components of $q$ are in the following spaces:
				\begin{align*}
					&u\in L^\infty(0,T;W^{1,4}(\Omega;\mathbb R^n)),\;u|_{D_T}=b|_{D_T},\\
					&c\in L^\infty(0,T;H^1(\Omega;\mathbb R^N))\cap H^1(0,T;L^2(\Omega;\mathbb R^N)),\; c(0)=c^0,\; c\in\Sigma\text{ a.e. in }\Omega_T,\\
					&z\in L^\infty(0,T;W_+^{1,p}(\Omega))\cap H^1(0,T;L^2(\Omega)),\;z(0)=z^0,\;\partial_t z\leq 0,
				\end{align*}
				and
				\begin{align*}
				\begin{aligned}
					&w\in L^2(0,T;H^1(\Omega;\mathbb R^N))&&\qquad\qquad\qquad\qquad\quad\;\;\text{for C-H systems,}\\
					&w\in L^2(\Omega_T;\mathbb R^N)&&\qquad\qquad\qquad\qquad\quad\;\;\text{for A-C systems}
				\end{aligned}
				\end{align*}
			\item 
				for all $\zeta\in H^1(\Omega;\mathbb R^N)$ and for a.e. $t\in[0,T]$:
				\begin{align}
				\label{eqn:weaksolution1}
					&\int_{\Omega}\partial_t c(t)\cdot\zeta\,\mathrm dx
						=
							\begin{cases}
								\int_{\Omega} \mathbb M\nabla w(t):\nabla\zeta\dx\qquad&\text{for C-H systems,}  \\
								\int_{\Omega} \mathbb Mw(t)\cdot\zeta\dx\qquad&\text{for A-H systems}
							\end{cases}
				\end{align}
			\item
				for all $\zeta\in H^1(\Omega;\mathbb R^N)$ and for a.e. $t\in[0,T]$:
				\begin{align}
						\int_{\Omega} w(t)\cdot\zeta\,\mathrm dx
							=&\int_{\Omega}\mathbb P\mathbf\Gamma\nabla c(t):\nabla\zeta+\mathbb P W_{,c}^\mathrm{ch,pol}(c(t))\cdot\zeta\,\mathrm dx\notag\\
						&+\int_\Omega\mathbb P W_{,c}^\mathrm{el}(e(u(t)),c(t),z(t))\cdot\zeta+\varepsilon\partial_t c(t)\cdot\zeta\,\mathrm dx
					\label{eqn:weaksolution2}
				\end{align}
			\item
				for all $\zeta\in W_D^{1,4}(\Omega;\mathbb R^n)$ and for a.e. $t\in[0,T]$:
				\begin{equation}
					\int_{\Omega} W_{,e}^\mathrm{el}(e(u(t)),c(t),z(t)):e(\zeta)
					+\varepsilon|\nabla u(t)|^2 \nabla u(t):\nabla \zeta\,\mathrm dx=0
				\label{eqn:weaksolution3}
				\end{equation}
			\item
				for all $\zeta\in W_-^{1,p}(\Omega)$ and for a.e. $t\in[0,T]$:
				\begin{align}
						&\int_{\Omega}(\varepsilon|\nabla z(t)|^{p-2}+1)\nabla z(t)\cdot\nabla\zeta
							+(W_{,z}^\mathrm{el}(e(u(t)),c(t),z(t))
							-\alpha+\beta(\partial_t z(t)))\zeta\,\mathrm dx\notag\\
						&\qquad\qquad\geq-\langle r(t),\zeta\rangle,
				\label{eqn:weaksolution4}
				\end{align}
				where $r(t)\in (W^{1,p}(\Omega))^\star$ satisfies
				$\langle r(t),z(t)-\zeta\rangle\geq 0$
				for all $\zeta\in W_+^{1,p}(\Omega)$
			\item
				energy inequality for a.e. $t\in[0,T]$:
				\begin{align}
						&\mathcal E_\varepsilon(u(t),c(t),z(t))-\mathcal E_\varepsilon(u^0,c^0,z^0)
							+\int_\Omega \alpha (z^0-z(t))\,\mathrm dx\notag\\
						&+\int_{\Omega_t} \beta |\partial_t z|^2
							+\varepsilon|\partial_t c|^2\,\mathrm dx\mathrm ds
							+\int_0^t\langle\mathcal S w(s),w(s)\rangle\,\mathrm ds\notag\\
						&\qquad\qquad\leq
							\int_{\Omega_t}W_{,e}^\mathrm{el}(e(u),c,z):e(\partial_t b)\,\mathrm dx\mathrm ds
							+\varepsilon\int_{\Omega_t}|\nabla u|^2\nabla u:\nabla\partial_t b\,\mathrm dx\mathrm ds,
				\label{eqn:weaksolution5}
				\end{align}
				where $u^0$ is the unique minimizer of $\mathcal E_\varepsilon(\cdot,c^0,z^0)$ in $W^{1,4}(\Omega;\mathbb R^n)$ with trace $u^0|_D=b(0)|_D$.
		\end{enumerate}
	\end{definition}
	With the help of the operator $\C S$, the diffusion equation \eqref{eqn:weaksolution1} can also be written as
	\begin{align*}
		&\int_{\Omega}\partial_t c(t)\cdot\zeta\,\mathrm dx
			=-\langle\mathcal Sw(t),\zeta\rangle,
	\end{align*}
	which will be used in the following.

	\begin{definition}[Weak solution for the limit system \eqref{eqn:unifyingModelClassical}]
	\label{def:weakSolutionLimit}
		A quadruple $q=(u,c,w,z)$ is called a
		weak solution of the system
		\eqref{eqn:unifyingModelClassical} with the initial-boundary conditions
		\eqref{eqn:unifyingModelIBC} if the following properties are satisfied:
		\begin{enumerate}
			\renewcommand{\labelenumi}{(\roman{enumi})}
			\item
				the components of $q$ are in the following spaces:
				\begin{align*}
					&u\in L^\infty(0,T;H^{1}(\Omega;\mathbb R^n)),\;u|_{D_T}=b|_{D_T},\\
					&c\in L^\infty(0,T;H^1(\Omega;\mathbb R^N)),\; c\in\Sigma\text{ a.e. in }\Omega_T,\\
					&z\in L^\infty(0,T;H_+^1(\Omega))\cap H^1(0,T;L^2(\Omega)),\;z(0)=z^0,\;\partial_t z\leq 0
				\end{align*}
				and
				\begin{align*}
				\begin{aligned}
					&w\in L^2(0,T;H^1(\Omega;\mathbb R^N))&&\qquad\qquad\qquad\;\;\text{for C-H systems,}\\
					&w\in L^2(\Omega_T;\mathbb R^N)&&\qquad\qquad\qquad\;\;\text{for A-C systems}
				\end{aligned}
				\end{align*}
			\item 
				for all $\zeta\in L^2(0,T;H^1(\Omega;\mathbb R^N))$ with $\partial_t\zeta\in L^2(\Omega_T;\mathbb R^N)$ and $\zeta(T)=0$:
				\begin{align*}
					&\int_{\Omega_T}(c-c^0)\cdot\partial_t\zeta\,\mathrm dx\,\mathrm dt
						=\int_0^T\langle\mathcal Sw,\zeta\rangle\,\mathrm dt
				\end{align*}
			\item
				for all $\zeta\in H^1(\Omega;\mathbb R^N)\cap L^\infty(\Omega;\mathbb R^N)$ and for a.e. $t\in[0,T]$:	
				\begin{equation*}
					\begin{split}
						\int_{\Omega} w(t)\cdot\zeta\,\mathrm dx
							=&\int_{\Omega}\mathbb P\mathbf\Gamma\nabla c(t):\nabla\zeta
							+\mathbb P W_{,c}^\mathrm{ch}(c(t))\cdot\zeta\,\mathrm dx\\
						&+\int_{\Omega}\mathbb P W_{,c}^\mathrm{el}(e(u(t)),c(t),z(t))\cdot\zeta\,\mathrm dx
					\end{split}
				\end{equation*}
			\item
				for all $\zeta\in H_D^1(\Omega;\mathbb R^n)$ and for a.e. $t\in[0,T]$:
				\begin{equation*}
					\int_{\Omega} W_{,e}^\mathrm{el}(e(u(t)),c(t),z(t)):e(\zeta)\,\mathrm dx=0
				\end{equation*}
			\item
				for all $\zeta\in H_-^1(\Omega)\cap L^\infty(\Omega)$ and for a.e. $t\in[0,T]$:
				\begin{align*}
						&\int_{\Omega}\nabla z(t)\cdot\nabla\zeta
							+(W_{,z}^\mathrm{el}(e(u(t)),c(t),z(t))
							-\alpha+\beta(\partial_t z(t)))\zeta\,\mathrm dx
							\geq-\langle r(t),\zeta\rangle,
				\end{align*}
				where $r(t)\in (H^1(\Omega)\cap L^\infty(\Omega))^\star$ satisfies
				$\langle r(t),z(t)-\zeta\rangle\geq 0$
				for all $\zeta\in H_+^1(\Omega)\cap L^\infty(\Omega)$
			\item
				energy inequality for a.e. $t\in[0,T]$:
				\begin{equation*}
					\begin{split}
						&\mathcal E(u(t),c(t),z(t))
							+\int_\Omega \alpha (z^0-z(t))\,\mathrm dx+\int_{\Omega_t} \beta |\partial_t z|^2\,\mathrm dx\mathrm ds
							+\int_0^t\langle\mathcal Sw(s),w(s)\rangle\,\mathrm ds\\
						&\qquad\qquad\leq\mathcal E(u^0,c^0,z^0)+
							\int_{\Omega_t}W_{,e}^\mathrm{el}(e(u),c,z):e(\partial_t b)\,\mathrm dx\mathrm ds,
					\end{split}
				\end{equation*}
				where $u^0$ is the unique minimizer of $\mathcal E(\cdot,c^0,z^0)$ in $H^1(\Omega;\mathbb R^n)$ with trace $u^0|_D=b(0)|_D$.
		\end{enumerate}
	\end{definition}
				
	Note that both notions of weak solution imply chemical mass conservation, i.e.
	\begin{align*}
		\int_\Omega c(t)\,\mathrm dx\equiv const.
	\end{align*}

\subsection{Assumptions and main results}
	The general setting, the growth assumptions and the assumptions on the coefficient tensors which are mandatory for the existence theorems are summarized below.
	\begin{enumerate}
		\renewcommand{\labelenumi}{(\roman{enumi})}
		\item \textit{Setting}
			\begin{align*}
			\begin{aligned}
				&\text{Space dimension}\hspace{4em}&&n \in \mathbb{N},\hspace{22.5em}\\
				&\text{Components in the alloy} &&N\in\mathbb N\text{ with }N\geq 2,\\
				&\text{Regularization exponent} &&p>n,\\
				&\text{Viscosity factors} &&\alpha,\beta>0,\\
				&\text{Domain} &&\Omega\subseteq\mathbb R^n\text{ bounded Lipschitz domain,}\\
				&\text{Dirichlet boundary} &&D\subseteq\partial\Omega\text{ with }\mathcal H^{n-1}(D)>0,\\
				&\text{Time interval} &&[0,T]\text{ with }T>0\\
			\end{aligned}
			\end{align*}
		\item \textit{Energy densities}
			\begin{align}
				\text{Elastic energy density}\qquad\;\;&W^\mathrm{el}\in \C C^1(\mathbb R^{n\times n}\times\mathbb R^N
					\times\mathbb R;\mathbb R_+)\text{ with}\notag\\
				\tag{A1}
				\label{eqn:growthEst1}
				&W^\mathrm{el}(e,c,z)=W^\mathrm{el}(e^t,c,z),\\
				\tag{A2}
				\label{eqn:growthEst2}
				&W^\mathrm{el}(e,c,z)\leq C (|e|^2+|c|^2+1),\\
				\tag{A3}
				\label{eqn:growthEst3}
				&\eta|e_1-e_2|^2\leq (W_{,e}^\mathrm{el}(e_1,c,z)-W_{,e}^\mathrm{el}(e_2,c,z)):(e_1-e_2),\\
				\tag{A4}
				\label{eqn:growthEst4}
				&|W_{,e}^\mathrm{el}(e_1+e_2,c,z)|\leq C(W^\mathrm{el}(e_1,c,z)+|e_2|+1),\\
				\tag{A5}
				\label{eqn:growthEst5}
				&|W_{,c}^\mathrm{el}(e,c,z)|\leq C (|e|^2+|c|^2+1),\\
				\tag{A6}
				\label{eqn:growthEst6}
				&|W_{,z}^\mathrm{el}(e,c,z)|\leq C (|e|^2+|c|^2+1),\\
				\text{Chemical energy densities}\quad&W^\mathrm{ch,pol},W^\mathrm{ch,log}\in \C C^1(\mathbb R^N;\mathbb R)\text{ with }W^\mathrm{ch,pol}\geq -C,\notag\\
				\tag{A7}
				\label{eqn:growthEst7}
				&|W_{,c}^\mathrm{ch,pol}(c)|\leq C(|c|^{2^\star/2}+1),\\
				\tag{A8}
				\label{eqn:growthEst8}
				&W^\mathrm{ch,log}(c)=\theta\sum_{k=1}^N c_k\log c_k+\frac 12 c\cdot Ac,\text{ }\theta>0\text{, }A\in\R_\mathrm{sym}^{n\times n}
			\end{align}
		\item \textit{Tensors}
			\begin{align*}
			\begin{aligned}
				&\text{Mobility tensor\hspace{3.7em}} &&\mathbb M\in\mathbb R^{N\times N}\text{ symmetric and positive definite on }T\Sigma\text{ and }\hspace{1.0em}\\
					&&&\sum_{l=1}^N \mathbb M_{kl}=0\text{ for all }k=1,\ldots,N,\\
				&\text{Energy gradient tensor} &&\mathbf\Gamma\in\mathcal L(\mathbb R^{N\times n};\mathbb R^{N\times n})\text{ symmetric and positive definite}\\
					&&&\text{fourth order tensor}
			\end{aligned}
			\end{align*}
	\end{enumerate}

	\begin{remark}
			Due to the effect of damage on the elastic response of
			the material, $W^\mathrm{el}$ is often modelled by 
			the following ansatz:
			\begin{equation*}
				\label{eq:Wel}
				W^\mathrm{el} = ( \Phi(z) + \tilde\eta )\, \hat{W}^\mathrm{el},
			\end{equation*} 
			where $\Phi:[0,1] \to \mathbb{R}_+$ is a continuously differentiable and monotonically
			increasing function with $\Phi(0)=0$ and $\tilde\eta >0$ is a small 
			value.
			
			A typically form of the {\it elastically stored energy density} 
			$\hat{W}^\mathrm{el}$ is as follows:
			\begin{equation}
			\label{eqn:inhomEnergy}
				\hat{W}^\mathrm{el}(c,e) = 
				\frac{1}{2} \big(e - e^*(c)\big) : \mathbb C(c) \big(e -e^*(c)\big). 
			\end{equation}
			Here, $e^*(c)$ denotes the {\it eigenstrain}, which is usually linear 
			in $c$, and  
			$\mathbb C(c)\in\mathcal L(\mathbb R_\mathrm{sym}^{n\times n})$ is a fourth order stiffness tensor, which is symmetric and positive definite. 
			The elastic energy density is called  {\it homogeneous} if the stiffness tensor does not depend on the concentration, i.\,e. 
			$\mathbb C(c)=\mathbb C$. 

			Note that the inhomogeneous elastic energy \eqref{eqn:inhomEnergy} fits into our setting
			with the previous growth assumptions \eqref{eqn:growthEst1}-\eqref{eqn:growthEst6}.
			In particular, we are not confined to homogeneous
			elasticity as in \cite{HeiKraus}.  There, the more restrictive growth condition
			$|W_{,c}^\mathrm{el}(e,c,z)|\leq C (|e|+|c|^2+1)$
			is used	instead of \eqref{eqn:growthEst5}.
	\end{remark}
	The main results of this work are summarized in the following theorems:
	\begin{theorem}[Existence theorem - polynomial case]
	\label{theorem:mainTheoremPoly}
		Let the above assumptions be satisfied. Then for every
		\begin{align*}
			&b\in W^{1,1}(0,T;W^{1,\infty}(\Omega;\mathbb R^n)),\\
			&c^0\in H^1(\Omega;\mathbb R^N)\text{ with }c^0\in\Sigma\text{ a.e. in }\Omega,\\
			&z^0\in H^1(\Omega)\text{ with }0\leq z^0\leq 1\text{ a.e. in }\Omega,
		\end{align*}
		there exists a weak solution $q$ of the system \eqref{eqn:unifyingModelClassical}
		with $W^\mathrm{ch}=W^\mathrm{ch,pol}$
		and the initial-boundary conditions \eqref{eqn:unifyingModelIBC}
		in the sense of Definition \ref{def:weakSolutionLimit}.
	\end{theorem}

	\begin{theorem}[Existence theorem - logarithmic case]
	\label{theorem:mainTheoremLog}
		Let the above assumptions be satisfied and, additionally, let $D=\partial\Omega$ and
		$\mathbf\Gamma=\gamma  \, \mathrm{Id}$ with a constant $\gamma>0$. Then for every
		\begin{align*}
			&b\in W^{1,1}(0,T;W^{1,\infty}(\Omega;\mathbb R^n)),\\
			&c^0\in H^1(\Omega;\mathbb R^N)\text{ with }c^0\in\Sigma\text{ and }c_k^0>0\text{ a.e. in }\Omega\text{  for $k=1, \dots, N$},\\
			&z^0\in H^1(\Omega)\text{ with }0\leq z^0\leq 1\text{ a.e. in }\Omega,
		\end{align*}
		there exists a weak solution $q$ of the
		system \eqref{eqn:unifyingModelClassical} with $W^\mathrm{ch}=W^\mathrm{ch,log}$
		and the initial-boundary conditions \eqref{eqn:unifyingModelIBC}
		in the sense of Definition \ref{def:weakSolutionLimit}.
		Additionally, $c_k>0$ a.e. in $\Omega_T$ for $k=1, \dots, N$.
	\end{theorem}

	\begin{remark}
		Note that for Theorem \ref{theorem:mainTheoremLog} the assumptions \eqref{eqn:growthEst2}, \eqref{eqn:growthEst5} and \eqref{eqn:growthEst6}
		can be replaced by
		\begin{align*}
			\tag{A2'}
			W^\mathrm{el}(e,c,z)\leq C (|e|^2+1),\\
			\tag{A5'}
			|W_{,c}^\mathrm{el}(e,c,z)|\leq C (|e|^2+1),\\
			\tag{A6'}
			|W_{,z}^\mathrm{el}(e,c,z)|\leq C (|e|^2+1),
		\end{align*}
		for all $c\in\R^N$ with $0\leq c_k\leq 1$ and $\sum_{k=1}^N c_k = 1$, all $e\in\R_\mathrm{sym}^{n\times n}$ and all $z\in\R$ with $0\leq z\leq 1$.
	\end{remark}


\section{Existence of weak solutions of \eqref{eqn:regUnifyingModelClassical}}
	The proof is based on \cite{HeiKraus}. Arguments similar to
	\cite{HeiKraus} are only sketched.

	Since $\varepsilon>0$ is fixed in this section, we omit the $\varepsilon$-dependence in the notation, e.g.
	$\mathcal E$ always means here $\mathcal E_\varepsilon$ and so on.
	Furthermore, $z^0$ is assumed to be in $W^{1,p}(\Omega)$ in this section.
	\begin{itemize}
		\item[\textsc{1.}]\textsc{Step: constructing time-discrete solutions.}
		
		Set $u^0$ to be a minimizer of $u\mapsto \mathcal E(u,c^0,z^0)$
		defined on the space $W^{1,4}(\Omega)$ with the constraint $u|_{D}=b(0)|_{D}$ in the sense of traces.
		
		Let the closed subspace $\mathcal Q_M^m$ of $H^1(\Omega;\mathbb R^n)\times H^1(\Omega;\mathbb R^N)\times
		W^{1,p}(\Omega)$ be defined by:
		\begin{align*}
			\mathcal Q_M^m=
			\left\{
			\begin{array}{l}
				u\in H^1(\Omega;\mathbb R^n),\\
				c\in H^1(\Omega;\mathbb R^N),\\
				z\in W^{1,p}(\Omega)
			\end{array}
			\Bigg|
			\begin{array}{l}
				u|_D=b(m\tau)|_{D},\\
				\int_\Omega c-c^0\,\mathrm dx=0\text{ for C-H systems,}\\
				0\leq z\leq z_M^{m-1}.
			\end{array}
			\right\}
		\end{align*}
		Based on the initial triple $(u^0,c^0,z^0)$, we construct $(u_M^m,c_M^m,z_M^m)$ for $m=1,\ldots,M$ recursively by minimizing
		the following functional $\mathbb E^m_M: \mathcal Q_M^m \to \mathbb R$:
		\begin{align}
			&\mathbb E^m_M(u,c,z) :=\tilde{\mathcal E}(u,c,z)
				+\tau\tilde{\mathcal R}\left(\frac{z-z_M^{m-1}}{\tau}\right)
				+\frac{\tau}{2}\left\|\frac{c-c_M^{m-1}}{\tau}\right\|_{X}^2
				+\frac{\varepsilon\tau}{2}\left\|\frac{c-c_M^{m-1}}{\tau}\right\|_{L^2}^2,
			\label{eqn:discreteE}
		\end{align}
		where $X$ denotes the space $\mathcal D$ (see \eqref{eqn:Dspace}) with the scalar-product
		\begin{align*}
			(c_1\,|\,c_2)_{X}:=\int_\Omega\mathbb M\nabla \mathcal S^{-1}c_1\cdot \nabla \mathcal S^{-1}c_2\dx
		\end{align*}
		for Cahn-Hilliard systems and $X=L^2(\Omega;\mathbb R^N)$ with the scalar-product
		\begin{align*}
			(c_1\,|\,c_2)_{X}:=\int_\Omega\mathbb M c_1\cdot c_2\dx
		\end{align*}
		for Allen-Cahn systems.

		Note that the last regularization term in \eqref{eqn:discreteE} is not necessary for Allen-Cahn equations due to the term
		with the $X$-norm. To use a uniform approach, we consider this term in both systems.
		By direct methods of calculus of variations the triple
		\begin{align*}
			(u_M^m,c_M^m,z_M^m):=\mathop{\mathrm{arg\,min}}_{(u,c,z)\in\mathcal Q_M^m}\;\mathbb E^m_M(u,c,z)
		\end{align*}
		exists, cf.~\cite{HeiKraus}. Furthermore, we set
		\begin{align*}
			w_M^m:=
			\begin{cases}
				-\mathcal S^{-1}\left(\frac{c_M^m-c_M^{m-1}}{\tau}\right)+\lambda_M^m&\text{for C-H systems},\\
				-\mathcal S^{-1}\left(\frac{c_M^m-c_M^{m-1}}{\tau}\right)&\text{for A-C systems},
			\end{cases}
		\end{align*}
		with the Lagrange multiplier $\lambda_M^m$ (associated with the mass constraint for C-H systems) given by
		\begin{align*}
			\lambda_M^m:=\dashint_\Omega W_{,c}^\mathrm{ch,pol}(c_M^m)+W_{,c}^\mathrm{el}(e(u_M^m),c_M^m,z_M^m)\,\mathrm dx.
		\end{align*}
		We define the time incremental solutions as 
                $$q_M^m:=(u_M^m,c_M^m,w_M^m,z_M^m)$$ 
                and introduce the piecewise constant interpolations
		$q_M$, $q_M^-,t_M,t_M^-$
		and the linear interpolation $\hat q_M$ as
		\begin{align*}
			t_M&:=\min\{m\tau\,|\,m\in\mathbb N_0\text{ and }m\tau\geq t\},\\
			t_M^-&:=\min\{(m-1)\tau\,|\,m\in\mathbb N_0\text{ and }m\tau\geq t\},\\
			q_M(t)&:=q^m_M\text{ for }t\in \big((m-1)\tau,m\tau\big],\\
			q_M^-(t)&:=q^m_M\text{ for }t\in \big[m\tau,(m+1)\tau\big),\\
			\hat q_M(t)&:=\beta q^m_M+(1-\beta)q^{m-1}_M\text{ for }t\in \big[(m-1)\tau,m\tau\big)
				\text{ and }\beta=\frac{t}{\tau}-(m-1).
		\end{align*}

		Due to the minimization properties of $(u_M^m,c_M^m,z_M^m)$, we establish the following variational formulas and energy estimate
		(cf. \cite[Lemma 6.2]{HeiKraus}):
		\begin{lemma}[Euler-Lagrange equation, energy estimate]
		\label{lemma:eulerLagrange}
			The functions $q_M$, $q_M^-$ and $\hat q_M$ satisfy the following properties for all $t\in(0,T)$:
			\begin{enumerate}
			\renewcommand{\labelenumi}{(\roman{enumi})}
				\item
					for all $\zeta\in H^1(\Omega;\mathbb R^N)$:
					\begin{equation}
						\int_{\Omega}(\partial_t\hat c_M(t))\cdot\zeta\,\mathrm dx
							=-\langle\mathcal Sw_M(t),\zeta\rangle
					\label{eqn:discreteSolution1}
					\end{equation}
				\item
					for all $\zeta\in H^1(\Omega;\mathbb R^N)$:
					\begin{align}
						\int_{\Omega} w_M(t)\cdot\zeta\,\mathrm dx
							={}&\int_{\Omega}\mathbb P\mathbf\Gamma\nabla c_M(t):\nabla\zeta+\mathbb PW_{,c}^\mathrm{ch,pol}(c_M(t))
							\cdot\zeta\,\mathrm dx\notag\\
						&+\int_\Omega \mathbb P W_{,c}^\mathrm{el}(e(u_M(t)),c_M(t),z_M(t))\cdot\zeta
							+\varepsilon\partial_t\hat c_M(t)\cdot\zeta\,\mathrm dx
					\label{eqn:discreteSolution2}
					\end{align}
				\item
					for all $\zeta\in W_D^{1,4}(\Omega;\mathbb R^n)$:
					\begin{equation}
						\int_{\Omega} W_{,e}^\mathrm{el}(e(u_M(t)),c_M(t),z_M(t)):e(\zeta)+\varepsilon|\nabla u_M(t)|^2\nabla u_M(t):
							\nabla\zeta\,\mathrm dx=0
					\label{eqn:discreteSolution3}
					\end{equation}
				\item
					for all $\zeta\in W^{1,p}(\Omega)$ with $0\leq\zeta+z_M(t)\leq z_M^-(t)$:
					\begin{align}
						&\int_{\Omega} (\varepsilon|\nabla z_M(t)|^{p-2}+1)\nabla z_M(t)\cdot\nabla\zeta
							+W_{,z}^\mathrm{el}(e(u_M(t)),c_M(t),z_M(t))\zeta\,\mathrm dx\notag\\
						&\qquad\qquad+\int_\Omega(-\alpha+\beta(\partial_t\hat z_M(t)))\zeta\,\mathrm dx\geq 0
					\label{eqn:discreteSolution4}
					\end{align}
				\item energy estimate:
					\begin{align}
						&\mathcal E(u_M(t),c_M(t),z_M(t))
							+\int_{0}^{t_M}\int_\Omega-\alpha\partial_t\hat z_M+\frac{\beta}{2}|\partial_t\hat z_M|^2
							+\frac{\varepsilon}{2}|\partial_t\hat c_M|^2\,\mathrm dx\mathrm ds\notag\\
						&+\int_0^{t_M}\frac 12\langle\mathcal Sw_M(s),w_M(s)\rangle\,\mathrm ds
							-\mathcal E(u^0,c^0,z^0)\notag\\
						&\qquad\qquad\leq
							\int_{0}^{t_M}\int_{\Omega}W_{,e}^\mathrm{el}(e(u_M^{-}+b-b_M^{-}),c_M^{-},
							z_M):e(\partial_t b)\,\mathrm dx\mathrm ds\notag\\
						&\qquad\qquad\quad
							+\varepsilon\int_{0}^{t_M}\int_{\Omega}|\nabla u_M^-+\nabla b
							-\nabla b_M^-|^2\nabla(u_M^-+b-b_M^-):\nabla\partial_t b\,\mathrm dx\mathrm ds.
					\label{eqn:energyEstimate}
					\end{align}
			\end{enumerate}
		\end{lemma}

			
		\item[\textsc{2.}]\textsc{Step: identifying convergent subsequences.}
			
			The energy estimate (v) in Lemma \ref{lemma:eulerLagrange}, growth condition \eqref{eqn:growthEst4} and a Gronwall estimation argument lead to
			a-priori estimates for the energy $\mathcal E(u_M(t),c_M(t),z_M(t))$ and for $\|\partial_t\hat z_M\|_{L^2(\Omega_T)}$,
			$\|\partial_t\hat c_M\|_{L^2(\Omega_T)}$ and $\int_0^{T}\langle\mathcal Sw_M(s),w_M(s)\rangle\,\mathrm ds$.
			By standard compactness arguments and a compactness theorem from Aubin and Lions
			\cite{Simon}, we deduce the following weak convergence
			properties, cf.~\cite{HeiKraus}:
			\begin{lemma}
			\label{lemma:weakConvergenceProperties}
				There exists a subsequence $\{ M_k \}$ and an element $q=(u,c,w,z)$ satisfying (i) from
				Definition \ref{def:weakSolutionRegularized} such that for a.e. $t\in[0,T]$:
				
				\begin{tabular}{ll}
					\begin{minipage}{19em}
						\begin{enumerate}
							\renewcommand{\labelenumi}{(\roman{enumi})}
							\item
								$u_{M_k}\stackrel{\star}{\rightharpoonup} u\text{ in }L^\infty(0,T;W^{1,4}(\Omega))$,
							\item
								$c_{M_k},c_{M_k}^-\stackrel{\star}{\rightharpoonup} c\text{ in }L^\infty(0,T;H^{1}(\Omega;\mathbb R^N))$,\\
								$c_{M_k}(t),c_{M_k}^-(t)\rightharpoonup c(t)\text{ in }H^{1}(\Omega;\mathbb R^N)$,\\
								$c_{M_k},c_{M_k}^-\rightarrow c\text{ a.e. in }\Omega_T$,\\
								$\hat c_{M_k}\rightharpoonup c\text{ in }H^1(0,T;L^2(\Omega;\mathbb R^N))$,
						\end{enumerate}
					\end{minipage}
					&
					\begin{minipage}{21em}
						\begin{enumerate}
							\renewcommand{\labelenumi}{(\roman{enumi})}
							\item[(iii)]
								$z_{M_k}, z^-_{M_k}\stackrel{\star}{\rightharpoonup} z\text{ in }L^\infty(0,T;W^{1,p}(\Omega))$,\\
								$z_{M_k}(t),z_{M_k}^-(t)\rightharpoonup z(t)\text{ in }W^{1,p}(\Omega)$,\\
								$z_{M_k},z_{M_k}^-\rightarrow z\text{ a.e. in }\Omega_T$,\\
								$\hat z_{M_k}\rightharpoonup z\text{ in } H^1(0,T;L^2(\Omega))$\\\\
						\end{enumerate}
					\end{minipage}
				\end{tabular}
				
				and\vspace{0.5em}
				
				\begin{tabular}{ll}
					\begin{minipage}{19em}
						\begin{enumerate}
							\renewcommand{\labelenumi}{(\roman{enumi})}
							\item[(iv)]
								$w_{M_k}\rightharpoonup w\text{ in } L^2(0,T;H^1(\Omega;\mathbb R^N))$\\
								$w_{M_k}\rightharpoonup w\text{ in } L^2(\Omega_T;\mathbb R^N)$
						\end{enumerate}
					\end{minipage}
					&
					\begin{minipage}{21em}
						\hspace{2.5em}for C-H systems,\\
						\hspace*{2.5em}for A-C systems
					\end{minipage}
				\end{tabular}
				
				as $k\rightarrow\infty$.
			\end{lemma}

			Exploiting the Euler-Lagrange equations, we can even prove stronger convergence properties.
			To proceed, we recall an approximation lemma from \cite{HeiKraus}.
			\begin{lemma}[{\cite[Lemma 5.2]{HeiKraus}}]
			\label{lemma:approximation}
				Let $q\geq 1$, $p>n$ and $f,\zeta\in L^q(0,T;W^{1,p}_+(\Omega))$ with
				$\{\zeta=0\}\supseteq\{f=0\}$.
				Furthermore, let $\{f_M\}_{M\in\mathbb N}\subseteq L^q(0,T;W^{1,p}_+(\Omega))$ be a sequence with
				$f_M(t)\rightharpoonup f(t)$ in $W^{1,p}(\Omega)$ as $M\rightarrow\infty$ for a.e. $t\in[0,T]$.
				Then there exists a sequence $\{\zeta_M\}_{M\in\mathbb N}\subseteq L^q(0,T;W^{1,p}_+(\Omega))$
				and constants $\nu_{M,t}>0$ such that
				\begin{enumerate}
					\renewcommand{\labelenumi}{(\roman{enumi})}
				\item
					$\zeta_M\rightarrow \zeta$ in $L^q(0,T;W^{1,p}(\Omega))$ as $M\rightarrow\infty$,
				\item
					$\zeta_M\leq \zeta$ a.e. in $\Omega_T$ for all $M\in\mathbb N$,
				\item
					$\nu_{M,t}\zeta_M(t)\leq f_M(t)$ a.e. in $\Omega$ for a.e. $t\in[0,T]$ and for all $M\in\mathbb N$.
				\end{enumerate}
				If, in addition, $\zeta\leq f$ a.e. in $\Omega_T$ then condition (iii) can be refined to
				\begin{enumerate}
					\renewcommand{\labelenumi}{(\roman{enumi})}
				\item[(iii)'] $\zeta_M\leq f_M$ a.e. in $\Omega_T$ for all $M\in\mathbb N$.
				\end{enumerate}
			\end{lemma}
			We are now able to prove strong convergence results by using uniform convexity estimates.
			\begin{lemma}[Strong convergence of the time incremental solutions]
			\label{lemma:convergenceProperties}
				There exists a subsequence $\{M_k\}$ such that for a.e. $t\in[0,T]$:
				
				\begin{tabular}{ll}
					\begin{minipage}{19.4em}
						\begin{enumerate}
							\renewcommand{\labelenumi}{(\roman{enumi})}
							\item
								$u_{M_k},u_{M_k}^-\rightarrow u$ in $L^4(0,T;W^{1,4}(\Omega;\mathbb R^n))$,\\
								$u_{M_k}(t),u_{M_k}^-(t)\rightarrow u(t)$ in $W^{1,4}(\Omega;\mathbb R^n)$,\\
								$u_{M_k},u_{M_k}^-\rightarrow u$ a.e. in $\Omega_T$,
							\item
								$c_{M_k},c_{M_k}^-\rightarrow c$ in $L^{2^\star}(0,T;H^1(\Omega;\mathbb R^N))$,\\
								$c_{M_k}(t),c_{M_k}^-(t)\rightarrow c(t)$ in $H^{1}(\Omega;\mathbb R^N)$,\\
								$c_{M_k},c_{M_k}^-\rightarrow c$ a.e. in $\Omega_T$,\\
								$\hat c_{M_k}\rightharpoonup c\text{ in }H^1(0,T;L^2(\Omega;\mathbb R^N))$,
						\end{enumerate}
					\end{minipage}
					&
					\begin{minipage}{21em}
						\begin{enumerate}
							\renewcommand{\labelenumi}{(\roman{enumi})}
							\item[(iii)]
								$z_{M_k},z_{M_k}^-\rightarrow z$ in $L^p(0,T;W^{1,p}(\Omega))$,\\
								$z_{M_k}(t),z_{M_k}^-(t)\rightarrow z(t)$ in $W^{1,p}(\Omega)$,\\
								$z_{M_k},z_{M_k}^-\rightarrow z$ a.e. in $\Omega_T$,\\
								$\hat z_{M_k}\rightharpoonup z$ in $H^1(0,T;L^2(\Omega))$\\\\\\\\\vspace{0.5em}
						\end{enumerate}
					\end{minipage}
				\end{tabular}
				
				as $k\rightarrow\infty$.
			\end{lemma}
			\begin{proof}
			  We omit the index $k$ in the proof.
				\begin{enumerate}
					\item[(i)] We refer to \cite[Lemma 5.9]{HeiKraus}.
					\item[(ii)]
						The weak convergence properties for $c_M$, $c_M^-$ and $\hat c_{M_k}$ follow from Lemma \ref{lemma:weakConvergenceProperties}.
						It remains to show strong convergence of $\nabla c_M$ to $\nabla c$ in $L^2(\Omega_T;\mathbb R^N)$.
						
						By the compact embedding $H^1(\Omega;\R^N)\hookrightarrow L^{2^\star/2+1}(\Omega;\mathbb R^N)$ and
						Lemma \ref{lemma:weakConvergenceProperties}, we get $\|c_M(t)-c(t)\|_{L^{2^\star/2+1}(\Omega;\R^N)}\rightarrow 0$ as $M\rightarrow \infty$
						for a.e. $t\in[0,T]$.
						The boundedness property $\esssup_{t\in[0,T]}\|c_M(t)-c(t)\|_{L^{2^\star/2+1}(\Omega;\R^N)}<C$ for all $M\in\N$ and Lebesgue's convergence theorem yield
						$c_M\rightarrow c$ as $M\rightarrow\infty$ in $L^{2^\star/2+1}(\Omega_T;\mathbb R^N)$.
						Testing \eqref{eqn:discreteSolution2} with $\zeta=c_M(t)$ and with
						$\zeta=c(t)$ gives after integration from $t=0$ to $t=T$:
						\begin{align*}
							\int_{\Omega_T} \mathbb P\mathbf\Gamma\nabla c_M:\nabla c_M\,\mathrm dx\mathrm dt
								={}&\int_{\Omega_T}w_M\cdot c_M-\mathbb PW_{,c}^\mathrm{ch,pol}(c_M)
								\cdot c_M\,\mathrm dx\mathrm dt\\
							&-\int_{\Omega_T} \mathbb P W_{,c}^\mathrm{el}(e(u_M),c_M,z_M)\cdot c_M
								+\varepsilon\partial_t\hat c_M\cdot c_M\,\mathrm dx\mathrm dt,\\
							\int_{\Omega_T} \mathbb P\mathbf\Gamma\nabla c_M:\nabla c\,\mathrm dx\mathrm dt
								={}&\int_{\Omega_T}w_M\cdot c-\mathbb PW_{,c}^\mathrm{ch,pol}(c_M)
								\cdot c\,\mathrm dx\mathrm dt\\
							&-\int_{\Omega_T} \mathbb P W_{,c}^\mathrm{el}(e(u_M),c_M,z_M)\cdot c
								+\varepsilon\partial_t\hat c_M\cdot c\,\mathrm dx\mathrm dt.
						\end{align*}
						Passing to $M\rightarrow\infty$ and comparing the right sides of the equations shows
						\begin{align*}
							\int_{\Omega_T} \mathbb P\mathbf\Gamma\nabla c_M:\nabla c_M\,\mathrm dx\mathrm dt
							\rightarrow
							\int_{\Omega_T} \mathbb P\mathbf\Gamma\nabla c:\nabla c\,\mathrm dx\mathrm dt.
						\end{align*}
						By using the properties $\mathbb P\nabla c_M=\nabla c_M$ and $\mathbb P\nabla c=\nabla c$, we eventually obtain
						$$
							\int_{\Omega_T}\mathbf\Gamma\nabla c_M:\nabla c_M\dx\dt\rightarrow
							\int_{\Omega_T}\mathbf\Gamma\nabla c:\nabla c\dx\dt.
						$$
						We end up with
						$$
							\int_{\Omega_T}\mathbf\Gamma(\nabla c_M-\nabla c):(\nabla c_M-\nabla c)\dx\dt\rightarrow 0.
						$$
						Therefore $\nabla c_M\rightarrow \nabla c$ in $L^2(\Omega_T;\mathbb R^N)$ since
						$\mathbf\Gamma$ is positive definite.
					\item[(iii)]
						Applying Lemma \ref{lemma:approximation} with $f=z$ and $f_M=z_M^-$ and $\zeta=z$ gives
						an approximation sequence $\{\zeta_M\}\subseteq L^p(0,T;W^{1,p}_+(\Omega))$ with
						the properties:
						\begin{subequations}
							\begin{align}
							\label{eqn:zApprox}
								&\zeta_M\rightarrow z\text{ in }L^p(0,T;W^{1,p}(\Omega)),\\
							\label{eqn:zApprox2}
								&0\leq \zeta_M\leq z_M^-\text{ for all }M\in\mathbb N.
							\end{align}
						\end{subequations}
						The estimate
						\begin{align*}
							C_\mathrm{uc}|\nabla z_M-\nabla z|^p\leq (|\nabla z_M|^{p-2}\nabla z_M-|\nabla z|^{p-2}
								\nabla z)\cdot \nabla(z_M-z)
						\end{align*}
						where $C_\mathrm{uc}>0$ is a constant
						and equation \eqref{eqn:discreteSolution4} tested with $\zeta=\zeta_M(t)-z_M(t)$ (possible due to \eqref{eqn:zApprox2})
						yield:
						\begin{align*}
							&C_\mathrm{uc}\int_{\Omega_T}\varepsilon|\nabla z_M-\nabla z|^p\,\mathrm dx\mathrm dt
								+\int_{\Omega_T}|\nabla z_M-\nabla z|^2\,\mathrm dx\mathrm dt\\
							&\qquad\leq \int_{\Omega_T} \big((\varepsilon|\nabla z_M|^{p-2}+1)\nabla z_M-(\varepsilon|\nabla z|^{p-2}+1)
								\nabla z\big)\cdot \nabla(z_M-z)\,\mathrm dx\mathrm dt\\
							&\qquad\leq \int_{\Omega_T} (\varepsilon|\nabla z_M|^{p-2}+1)\nabla z_M
								\cdot \nabla(z_M-\zeta_M)\,\mathrm dx\mathrm dt\\
							&\qquad\quad +\int_{\Omega_T} (\varepsilon|\nabla z_M|^{p-2}+1)\nabla z_M\cdot \nabla(\zeta_M-z)
								-(\varepsilon|\nabla z|^{p-2}+1)\nabla z\cdot \nabla(z_M-z)
								\,\mathrm dx\mathrm dt\\
							&\qquad\leq \int_{\Omega_T}(W_{,z}^\mathrm{el}(e(u_M),c_M,z_M)-\alpha
								+\beta\partial_t \hat z_M)(\zeta_M-z_M)\,\mathrm dx\mathrm dt\\
							&\qquad\quad +\int_{\Omega_T} (\varepsilon|\nabla z_M|^{p-2}+1)\nabla z_M\cdot \nabla(\zeta_M-z)
								-(\varepsilon|\nabla z|^{p-2}+1)\nabla z\cdot \nabla(z_M-z)\,\mathrm dx\mathrm dt\\
							&\qquad\leq \underbrace{\|W_{,z}^\mathrm{el}(e(u_M),c_M,z_M)
								-\alpha+\beta\partial_t \hat z_M\|_{L^2(\Omega_T)}
								}_{\text{bounded}}\|\zeta_M-z_M\|_{L^2(\Omega_T)}\\
							&\qquad\quad
								+\underbrace{(\varepsilon\|\nabla z_M\|_{L^p(\Omega_T)}^{p-1}
								+\|\nabla z_M\|_{L^{p/(p-1)}(\Omega_T)})}_{\text{bounded}}
								\|\nabla\zeta_M-\nabla z\|_{L^p(\Omega_T)}\\
							&\qquad\quad
								-\int_{\Omega_T}(\varepsilon|\nabla z|^{p-2}+1)\nabla z\cdot \nabla(z_M-z)\,\mathrm dx\mathrm dt
						\end{align*}
						Due to \eqref{eqn:zApprox} and $z_M\stackrel{\star}{\rightharpoonup}z$ in $L^\infty(0,T;W^{1,p}(\Omega))$
						as well as $z_M\rightarrow z$ in $L^2(\Omega_T)$,
						each term on the right hand side converges to $0$ as $M\rightarrow\infty$.\ep
				\end{enumerate}
			\end{proof}\\
		      \item[\textsc{3.}]\textsc{Step: establishing a precise energy inequality.}
		
			In this step we establish an asymptotic energy inequality, which is sharper than the energy inequality
			in \eqref{eqn:energyEstimate}. Note, that compared to \eqref{eqn:energyEstimate} the factor $1/2$ in front of
			$\langle\mathcal Sw_M(s),w_M(s)\rangle$ is missing.
			To simplify notation, we omit the index $k$ in the following.
			\begin{lemma}
				\label{lemma:preciseEnergyEstimate}
				For every $t\in[0,T]$:
				\begin{align*}
					&\mathcal E(u_M(t),c_M(t),z_M(t))
						+\int_{0}^{t_M}\int_\Omega-\alpha\partial_t\hat z_M+\beta|\partial_t\hat z_M|^2
						+\varepsilon|\partial_t\hat c_M|^2\,\mathrm dx\mathrm ds\\
					&+\int_0^{t_M}\langle\mathcal Sw_M(s),w_M(s)\rangle\,\mathrm ds-\mathcal E(u^0,c^0,z^0)\\
					&\qquad\qquad\leq
						\int_{0}^{t_M}\int_{\Omega}W_{,e}^\mathrm{el}(e(u_M^{-}+b-b_M^{-}),c_M^{-},
						z_M):e(\partial_t b)\,\mathrm dx\mathrm ds\\
					&\qquad\qquad\quad
						+\varepsilon\int_{0}^{t_M}\int_{\Omega}|\nabla u_M^-+\nabla b
						-\nabla b_M^-|^2\nabla(u_M^-+b-b_M^-):\nabla\partial_t b\,\mathrm dx\mathrm ds+\kappa_M
				\end{align*}
				with $\kappa_M\rightarrow 0$ as $M\rightarrow\infty$.
			\end{lemma}
			\begin{proof}
				Applying the estimate $\mathbb E_M^m(q_M^m)\leq \mathbb E_M^m(u_M^{m-1}+b_M^m-b_M^{m-1},c_M^{m},z_M^{m})$
				for $m=1$ to $\frac{t_M}{\tau}$ yields
				(cf. \cite[Lemma 6.10]{HeiKraus}):
				\begin{align}
					&\mathcal E(u_M(t),c_M(t),z_M(t))-\mathcal E(u^0,c^0,z^0)\notag\\
					&\qquad\leq\varepsilon\int_{0}^{t_M}\int_\Omega
						|\nabla(u_M^-+b(s)-b_M^-)|^2\nabla(u_M^-+b(s)-b_M^-):\nabla\partial_t b(s)\,\mathrm dx\mathrm ds\notag\\
					&\qquad\quad
						+\int_{0}^{t_M}\int_\Omega W_{,e}^\mathrm{el}(e(u_M^-+b-b_M^-),c_M^-,z_M^-):
						e(\partial_t b)\,\mathrm dx\mathrm ds\notag\\
					&\qquad\quad
						+\underbrace{\int_{0}^{t_M}\int_\Omega
						W_{,c}^\mathrm{el}(e(u_M^-+b_M-b_M^-),\hat c_M,z_M^-)\cdot
						\partial_t \hat c_M\,\mathrm dx\mathrm ds}_{(\star)_1}\notag\\
					&\qquad\quad+\underbrace{\int_{0}^{t_M}\int_\Omega\mathbf\Gamma\nabla \hat c_M: \nabla \partial_t\hat c_M
						+W_{,c}^\mathrm{ch,pol}(\hat c_M)\cdot\partial_t\hat c_M\,\mathrm dx\mathrm ds}_{(\star)_2}\notag\\
					&\qquad\quad+\underbrace{\int_{0}^{t_M}\int_\Omega
						W_{,z}^\mathrm{el}(e(u_M^-+b_M-b_M^-),c_M,\hat z_M)\partial_t\hat z_M\,\mathrm dx\mathrm ds}_{(\star\star)_1}\notag\\
					&\qquad\quad+\underbrace{\int_{0}^{t_M}\int_\Omega\varepsilon|\nabla \hat z_M|^{p-2}\nabla \hat z_M
						\cdot \nabla \partial_t\hat z_M+\nabla\hat z_M\cdot\nabla\partial_t\hat z_M\,\mathrm dx\mathrm ds}_{(\star\star)_2}.
				\label{eqn:preciseEI1}
				\end{align}
				The elementary inequalities
				\begin{align*}
					(|\nabla \hat z_M|^{p-2}\nabla \hat z_M-|\nabla z_M|^{p-2}\nabla
					z_M)\cdot\nabla\partial_t\hat z_M\leq 0\quad\text{ and }\quad
					(\nabla \hat z_M-\nabla z_M)\cdot\nabla\partial_t\hat z_M\leq 0
				\end{align*}
				and \eqref{eqn:discreteSolution4} tested with $\zeta:=-\partial_t\hat z_M(t)\tau$ lead to the estimate:
				\begin{equation*}
					\begin{split}
						&(\star\star)_1+(\star\star)_2\\
						&\qquad\leq-\int_{0}^{t_M}\int_\Omega
							-\alpha\partial_t\hat z_M+\beta|\partial_t\hat z_M|^2\,\mathrm dx\mathrm ds\\
						&\qquad\quad+\underbrace{\int_{0}^{t_M}\int_\Omega
							(W_{,z}^\mathrm{el}(e(u_M^-+b_M-b_M^-),c_M,\hat z_M)-W_{,z}^\mathrm{el}(e(u_M),c_M,z_M))
							\partial_t\hat z_M\,\mathrm dx\mathrm ds}_{=:\kappa_M^3}.
					\end{split}
				\end{equation*}
				Furthermore,
				\begin{equation*}
					\begin{split}
						(\star)_1\leq{}&\int_{0}^{t_M}\int_\Omega
							W_{,c}^\mathrm{el}(e(u_M),c_M,z_M)\cdot\partial_t \hat c_M\,\mathrm dx\mathrm ds\\
						&+\underbrace{\int_{0}^{t_M}\int_\Omega
							(W_{,c}^\mathrm{el}(e(u_M^-+b_M-b_M^-),\hat c_M,z_M^-)-W_{,c}^\mathrm{el}(e(u_M),c_M,z_M))\cdot
							\partial_t\hat c_M\,\mathrm dx\mathrm ds}_{=:\kappa_M^1}.
					\end{split}
				\end{equation*}
				Using the elementary estimate
				$\mathbf\Gamma(\nabla \hat c_M-\nabla
                                c_M):\nabla\partial_t\hat c_M\leq 0$ gives
				\begin{equation*}
					\begin{split}
						(\star)_2\leq{}&
							\int_{0}^{t_M}\int_\Omega
							\mathbf\Gamma\nabla c_M:\nabla \partial_t\hat c_M
							+W_{,c}^\mathrm{ch,pol}(c_M)\cdot\partial_t\hat c_M\,\mathrm dx\mathrm ds\\
						&+\underbrace{\int_{0}^{t_M}\int_\Omega
							(W_{,c}^\mathrm{ch,pol}(\hat c_M)-W_{,c}^\mathrm{ch,pol}(c_M))\cdot\partial_t\hat c_M\,\mathrm dx\mathrm ds
							}_{=:\kappa_M^2}.
					\end{split}
				\end{equation*}
				Hence, applying equations \eqref{eqn:discreteSolution2} with $\zeta=\partial_t \hat c_M(t)$ and
				\eqref{eqn:discreteSolution1} with $\zeta=w_M(t)$ by noticing $\mathbb P\partial_t\hat c_M(t)=\partial_t\hat c_M(t)$ shows
				\begin{equation*}
					\begin{split}
						(\star)_1+(\star)_2
							\leq -\int_{0}^{t_M}\langle\mathcal Sw_M(s),w_M(s)\rangle\,\mathrm ds
							-\int_{0}^{t_M}\int_\Omega \varepsilon|\partial_t\hat c_M|^2\,\mathrm dx\mathrm ds
							+\kappa_M^1+\kappa_M^2.
					\end{split}
				\end{equation*}
				Lebesgue's generalized convergence theorem, growth conditions \eqref{eqn:growthEst5}-\eqref{eqn:growthEst7}
				and Lemma \ref{lemma:convergenceProperties}
				show $\kappa_M:=\kappa_M^1+\kappa_M^2+\kappa_M^3\rightarrow 0$ as $M\rightarrow\infty$.
				We would like to emphasize that we need the boundedness of $\nabla u_M$ in $L^4(\Omega_T;\mathbb R^{n\times n})$
				and the boundedness of $\partial_t \hat{c}_M$ and $\partial_t \hat{z}_M$ in $L^2(\Omega_T)$ with respect to $M$.\ep
			\end{proof}

		\item[\textsc{4.}]\textsc{Step: passing to $M\rightarrow \infty$.}
			Using Lemma \ref{lemma:weakConvergenceProperties}, Lemma \ref{lemma:convergenceProperties} and \eqref{eqn:discreteSolution1}, \eqref{eqn:discreteSolution2}
			and \eqref{eqn:discreteSolution3} we establish (ii), (iii) and (iv) of Definition \ref{def:weakSolutionRegularized}.
			Moreover, Lemma \ref{lemma:preciseEnergyEstimate} implies
			\begin{align*}
				&\mathcal E(u_M(t),c_M(t),z_M(t))
					+\int_{\Omega_t}-\alpha\partial_t\hat z_M+\beta|\partial_t\hat z_M|^2
					+\varepsilon|\partial_t\hat c_M|^2\,\mathrm dx\mathrm ds\\
				&+\int_0^{t}\langle\mathcal Sw_M(s),w_M(s)\rangle\,\mathrm ds-\mathcal E(u^0,c^0,z^0)\\
				&\qquad\qquad\leq
					\int_{0}^{t_M}\int_{\Omega}W_{,e}^\mathrm{el}(e(u_M^{-}+b-b_M^{-}),c_M^{-},
					z_M):e(\partial_t b)\,\mathrm dx\mathrm ds\\
				&\qquad\qquad\quad
					+\varepsilon\int_{0}^{t_M}\int_{\Omega}|\nabla u_M^-+\nabla b
					-\nabla b_M^-|^2\nabla(u_M^-+b-b_M^-):\nabla\partial_t b\,\mathrm dx\mathrm ds+\kappa_M.
			\end{align*}
			The energy estimate (vi) from Definition \ref{def:weakSolutionRegularized} follows from above by using the known convergence properties
			and weakly semi-continuity arguments.
			
			It remains to show (v) of Definition \ref{def:weakSolutionRegularized}.
			To proceed, we cite the following lemma from \cite{HeiKraus} which provides a tool to drop a restriction on the space of test-functions
			for a variational inequality of a specific form.
			\begin{lemma}[{\cite[Lemma 5.3]{HeiKraus}}]
			\label{lemma:preciseLowerBound}
				Let $p>n$ and $f\in L^{p/(p-1)}(\Omega;\mathbb R^n)$, $g\in L^1(\Omega)$, $z\in W_+^{1,p}(\Omega)$
				with $z\geq 0$, $f\cdot\nabla z\geq 0$ and $\{f=0\}\supseteq\{z=0\}$ a.e.. Furthermore, we assume that
				\begin{align*}
					\int_\Omega f\cdot\nabla\zeta+g\zeta\,\mathrm dx\geq 0\quad
						\text{for all }\zeta\in W_-^{1,p}(\Omega)\text{ with }\{\zeta=0\}\supseteq \{z=0\}.
				\end{align*}
				Then
				\begin{align*}
					\int_\Omega f\cdot\nabla\zeta+g\zeta\,\mathrm dx
						\geq\int_{\{z=0\}}[g]^+ \zeta\,\mathrm dx\quad
						\text{for all }\zeta\in W_-^{1,p}(\Omega).
				\end{align*}
			\end{lemma}
			
			We are now able to prove the remaining property.
			\begin{lemma}
			\label{lemma:VI}
				We have
				\begin{align}
						&\int_{\Omega}(\varepsilon|\nabla z(t)|^{p-2}+1)\nabla z(t)\cdot\nabla\zeta
							+(W_{,z}^\mathrm{el}(e(u(t)),c(t),z(t))
							-\alpha+\beta(\partial_t z(t)))\zeta\,\mathrm dx\notag\\
						&\qquad\qquad\geq-\langle r(t),\zeta\rangle,
					\label{eqn:VIproof}
				\end{align}
				for all $\zeta\in W_-^{1,p}(\Omega)$ and for a.e. $t\in[0,T]$, where $r(t)\in L^1(\Omega)\subseteq (W^{1,p}(\Omega))^\star$ is given by
				\begin{align}
					r(t):=-\chi_{\{z(t)=0\}}[W_{,z}^\mathrm{el}(e(u(t)),c(t),z(t))]^+.
					\label{eqn:r_proof}
				\end{align}
			\end{lemma}
			\begin{proof}
				First of all, we take any test-function $\zeta\in L^p(0,T;W_-^{1,p}(\Omega))$ with\linebreak
				$\{\zeta=0\}\supseteq\{z=0\}$. Lemma \ref{lemma:approximation} gives a sequence
				$\{\zeta_M\}\subseteq L^p(0,T;W_-^{1,p}(\Omega))$ with
				$\zeta_M\rightarrow\zeta$ in $L^p(0,T;W^{1,p}(\Omega))$ and $0\geq \nu \zeta_M(t)\geq -z_M(t)$
				where $\nu$ depends on $M$ and $t$.
				Therefore \eqref{eqn:discreteSolution4} holds for $\zeta=\zeta_M(t)$.
				Integration from $0$ to $T$ and passing to $M\rightarrow\infty$ gives
				\begin{align*}
					&\int_{\Omega_T} (\varepsilon|\nabla z|^{p-2}+1)\nabla z\cdot\nabla\zeta
						+(W_{,z}^\mathrm{el}(e(u),c,z)
						-\alpha+\beta(\partial_t z))\zeta\,\mathrm dx\mathrm dt\geq 0.
				\end{align*}
				In other words,
				\begin{align*}
					&\int_{\Omega} (\varepsilon|\nabla z(t)|^{p-2}+1)\nabla z(t)\cdot\nabla\zeta
						+W_{,z}^\mathrm{el}(e(u(t)),c(t),z(t))\zeta\,\mathrm dx\\
					&\qquad\qquad+\int_\Omega(-\alpha+\beta(\partial_t  z(t)))\zeta\,\mathrm dx\geq 0
				\end{align*}
				holds for every $\zeta\in W_-^{1,p}(\Omega)$ with $\{\zeta=0\}\supseteq\{z(t)=0\}$ and a.e. $t\in[0,T]$.
				To finish the proof, we need to extend the variational inequality to the whole space
				$W_-^{1,p}(\Omega)$.
				
				Setting $f=(\varepsilon|\nabla z(t)|^{p-2}+1)\nabla z(t)$ and
				$g=W_{,z}^\mathrm{el}(e(u(t)),c(t),z(t))-\alpha+\beta(\partial_t z(t))$, Lemma
				\ref{lemma:preciseLowerBound} shows for every $\zeta\in W_-^{1,p}(\Omega)$
				\begin{align*}
					&\int_{\Omega} (\varepsilon|\nabla z(t)|^{p-2}+1)\nabla z(t)\cdot\nabla\zeta
						+(W_{,z}^\mathrm{el}(e(u(t)),c(t),z(t))
						-\alpha+\beta(\partial_t z(t)))\zeta\,\mathrm dx\\
					&\qquad\qquad\geq \int_{\{z(t)=0\}}[W_{,z}^\mathrm{el}(e(u(t)),c(t),z(t))-\alpha+\beta(\partial_t z(t))]^+
						\zeta\,\mathrm dx\\
					&\qquad\qquad\geq \int_{\{z(t)=0\}}[W_{,z}^\mathrm{el}(e(u(t)),c(t),z(t))]^+\zeta\,\mathrm dx.
				\end{align*}
				Now, variational inequality \eqref{eqn:VIproof} follows by setting
				\begin{align*}
					r(t):=-\chi_{\{z(t)=0\}}[W_{,z}^\mathrm{el}(e(u(t)),c(t),z(t))]^+.
				\end{align*}
				\ep
			\end{proof}
			\begin{remark}
				Lemma \ref{lemma:VI} gives more information than (v) from Definition \ref{def:weakSolutionRegularized}. It provides a special choice for
				$r(t)$ given by \eqref{eqn:r_proof}.
			\end{remark}
	\end{itemize}

\section{Existence of weak solutions of \eqref{eqn:unifyingModelClassical} - polynomial case}
	In this chapter, we show that an appropriate subsequence of
	the regularized solutions $q_\varepsilon$ for $\varepsilon\in(0,1]$ of Definition \ref{def:weakSolutionRegularized} converges in ``some sense''
	to $q$ which satisfies the limit equations given in Definition \ref{def:weakSolutionLimit}.
	Besides that the initial damage profile $z^0$ in this chapter is in $H^1(\Omega)$.
	We approximate $z^0\in H^1(\Omega)$ by a sequence $\{z_\varepsilon^0\}$ in $W^{1,p}(\Omega)$ such that $z_\varepsilon^0\rightarrow z^0$
	in $H^1(\Omega)$ as $\varepsilon\searrow 0$.
	Using the energy inequality and Gronwall's inequality, we establish again the following energy estimate.
	\begin{lemma}
		We have
		\begin{align*}
			&\mathcal E_\varepsilon(u_\varepsilon(t),c_\varepsilon(t),z_\varepsilon(t))
				+\int_0^t\int_\Omega-\alpha\partial_t z_\varepsilon+\beta|\partial_t z_\varepsilon|^2
				+\varepsilon|\partial_t c_\varepsilon|^2\,\mathrm dx\,\mathrm ds
				+\int_0^t\langle\mathcal Sw_\varepsilon(s),w_\varepsilon(s)\rangle\,\mathrm ds\\
			&\qquad\qquad\leq C(\mathcal E_\varepsilon(u_\varepsilon^0,c^0,z_\varepsilon^0)+1)
		\end{align*}
		for a.e. $t\in [0,T]$ and every $\varepsilon\in(0,1]$.
	\end{lemma}

	Since $\mathcal E_\varepsilon(u_\varepsilon^0,c^0,z_\varepsilon^0)
	\leq\mathcal E_\varepsilon(u_1^0,c^0,z_\varepsilon^0)\leq \mathcal E_1(u_1^0,c^0,z_\varepsilon^0)$, the left hand side is
	also uniformly bounded with respect to a.e. $t\in [0,T]$ and every $\varepsilon\in(0,1]$.
	By using standard compactness theorems and uniform convexity properties of $W^\mathrm{el}$ (see \eqref{eqn:growthEst3}),
	we obtain the following convergence properties (cf. \cite{HeiKraus}).
	\begin{lemma}[Convergence properties of $q_\varepsilon$]
	\label{lemma:convergenceProperties2}
		There exists a subsequence $\{\varepsilon_k\}$ with $\varepsilon_k\searrow 0$ as $k\rightarrow\infty$ and
		an element $q=(u,c,w,z)$ satisfying (i) of Definition \ref{def:weakSolutionLimit} such that for a.e. $t\in[0,T]$
		
		\begin{tabular}{ll}
			\begin{minipage}{20em}
				\begin{enumerate}
					\renewcommand{\labelenumi}{(\roman{enumi})}
					\item
						$u_{\varepsilon_k}\rightarrow u$ in $L^2(0,T;H^1(\Omega;\mathbb R^n))$,\\
						$\sqrt[3]{\varepsilon_k} \nabla u_{\varepsilon_k}\rightarrow 0$ in $L^\infty(0,T;L^4(\Omega;\mathbb R^n))$,\\
						$u_{\varepsilon_k}(t)\rightarrow u(t)$ in $H^{1}(\Omega;\mathbb R^n)$,\\
						$u_{\varepsilon_k}\rightarrow u$ a.e. in $\Omega_T$,\\
						$u_{\varepsilon_k}^0\rightarrow u^0$ in $H^{1}(\Omega;\mathbb R^n)$,\\
						$\sqrt[3]{\varepsilon_k} \nabla u_{\varepsilon_k}^0\rightarrow 0$ in $L^4(\Omega;\mathbb R^n)$,
					\item
						$c_{\varepsilon_k}\stackrel{\star}{\rightharpoonup} c$ in $L^\infty(0,T;H^1(\Omega;\mathbb R^N))$,\\
						$\varepsilon_k \partial_t c_{\varepsilon_k}\rightarrow 0$ in $L^2(\Omega_T;\mathbb R^N)$,\\
						$c_{\varepsilon_k}(t)\rightharpoonup c(t)$ in $H^{1}(\Omega;\mathbb R^N)$,\\
						$c_{\varepsilon_k}\rightarrow c$ a.e. in $\Omega_T$,\\
				\end{enumerate}
			\end{minipage}
			&
			\begin{minipage}{21em}
				\begin{enumerate}
					\renewcommand{\labelenumi}{(\roman{enumi})}
					\item[(iii)]
						$z_{\varepsilon_k}\stackrel{\star}{\rightharpoonup} z$ in $L^\infty(0,T;H^1(\Omega))$,\\
						$\sqrt[p-1]{\varepsilon_k} \nabla z_{\varepsilon_k}\rightarrow 0$ in $L^\infty(0,T;L^p(\Omega))$,\\
						$z_{\varepsilon_k}(t)\rightharpoonup z(t)$ in $H^1(\Omega)$,\\
						$z_{\varepsilon_k}\rightarrow z$ a.e. in $\Omega_T$,\\
						$z_{\varepsilon_k}\rightharpoonup z$ in $H^1(0,T;L^2(\Omega))$
						\\\\\\\\\\\\\\
				\end{enumerate}
			\end{minipage}
		\end{tabular}
		
		as $k\rightarrow\infty$.
		We additionally obtain for Cahn-Hilliard systems
		$$w_{\varepsilon_k}\rightharpoonup w\text{ in }L^2(0,T;H^1(\Omega;\mathbb R^N))$$
		and for Allen-Cahn systems
		\begin{align*}
			&w_{\varepsilon_k}\rightharpoonup w\text{ in }L^2(\Omega_T;\mathbb R^N),\\
			&c_{\varepsilon_k}\rightharpoonup c\text{ in }H^1(0,T;L^2(\Omega;\mathbb R^N))
		\end{align*}
		as $k\rightarrow\infty$.
	\end{lemma}
	As before, we will omit the index $k$ in the subscripts below.
	\begin{remark}
		We would like to mention that the arguments in \cite[Lemma 6.14]{HeiKraus} cannot be adapted to prove strong convergence
		properties of $\nabla c_\varepsilon$ and $\nabla z_\varepsilon$ due to
		the more generous growth condition \eqref{eqn:growthEst5} as well as
		the use of Lemma \ref{lemma:approximation}
		where the compact embedding $W^{1,p}(\Omega)\hookrightarrow \C C^{0,\alpha}(\overline{\Omega})$ for $p>n$ with
		$\alpha>0$ and $\alpha<1-\frac np$ is exploited.
	\end{remark}
	We are now able to establish existence of weak solutions of \eqref{eqn:unifyingModelClassical}
	in the polynomial case.
	
	\begin{proof}[Proof of Theorem \ref{theorem:mainTheoremPoly}]
		Whenever we refer in the following to \eqref{eqn:weaksolution1}-\eqref{eqn:weaksolution5} the functions $u,c,w,z$ and $r$ are substituted by
		$u_\varepsilon,c_\varepsilon,w_\varepsilon,z_\varepsilon$ and $r_\varepsilon$.
		Moreover, Lemma \ref{lemma:convergenceProperties2} is used without mention in the following.
		\begin{enumerate}
			\renewcommand{\labelenumi}{(\roman{enumi})}
			\item Let $\zeta\in L^2(0,T;H^1(\Omega;\mathbb R^N))$ with $\partial_t\zeta\in L^2(\Omega_T;\mathbb R^N)$ and $\zeta(T)=0$.
				Integration from $t=0$ to $t=T$ of \eqref{eqn:weaksolution1} and integration by parts yield
				\begin{align*}
					\int_{\Omega_T}(c_\varepsilon-c^0)\cdot\partial_t\zeta\,\mathrm dx\mathrm ds
						=\int_0^T\langle\mathcal Sw_\varepsilon,\zeta\rangle\,\mathrm ds.
				\end{align*}
				Passing to $\varepsilon\searrow 0$ shows (ii) of Definition \ref{def:weakSolutionLimit}.
			\item
				Let $\zeta\in L^2(0,T;H^1(\Omega;\mathbb R^N))\cap L^\infty(\Omega_T;\mathbb R^N)$.
				Integration from $t=0$ to $t=T$ of \eqref{eqn:weaksolution2} and passing to $\varepsilon\searrow 0$ yield
				\begin{align*}
					\int_{\Omega_T}w\cdot\zeta\,\mathrm dx\mathrm ds
						=\int_{\Omega_T}\mathbb P\mathbf\Gamma\nabla c:\nabla\zeta+(\mathbb P W_{,c}^\mathrm{ch,pol}(c)
						+\mathbb P W_{,c}^\mathrm{el}(e(u),c,z))\cdot\zeta\,\mathrm dx\mathrm ds.
				\end{align*}
				Note that
				\begin{align*}
					\left|\int_{\Omega_T}\varepsilon\partial_t c_\varepsilon\cdot\zeta\,\mathrm dx\mathrm ds\right|
					\leq \varepsilon \|\partial_t c_\varepsilon\|_{L^2(\Omega_T;\mathbb R^N)}\|\zeta\|_{L^2(\Omega_T;\mathbb R^N)}\rightarrow 0
				\end{align*}
				as $\varepsilon\searrow 0$.
				This shows (iii) of Definition \ref{def:weakSolutionLimit} with $W_{,c}^\mathrm{ch}=W_{,c}^\mathrm{ch,pol}$.
			\item
				Let $\zeta\in W_D^{1,4}(\Omega;\mathbb R^n)$ be arbitrary.
				Passing to $\varepsilon\searrow 0$ in \eqref{eqn:weaksolution3} yields for a.e. $t\in[0,T]$
				\begin{align}
					\int_{\Omega} W_{,e}^\mathrm{el}(e(u(t)),c(t),z(t)):e(\zeta)\,\mathrm dx=0,
				\label{eqn:limitQuasistaticEqui}
				\end{align}
				by noticing
				\begin{align*}
					\left|\int_\Omega\varepsilon|\nabla u_\varepsilon(t)|^2 \nabla u_\varepsilon(t):\nabla \zeta\,\mathrm dx\right|
						\leq \varepsilon \|\nabla u_\varepsilon(t)\|_{L^4(\Omega)}^3\|\zeta\|_{L^4(\Omega)}\rightarrow 0.
				\end{align*}
				A density argument shows that \eqref{eqn:limitQuasistaticEqui} also holds for all
				$\zeta\in H_D^{1}(\Omega;\mathbb R^n)$.
				Therefore, (iv) of Definition \ref{def:weakSolutionLimit} is shown.
			\item
				The characteristic functions $\chi_{\{z_\varepsilon=0\}}$ are bounded in $L^\infty(\Omega_T)$ with respect to $\varepsilon\in(0,1]$.
				We select a subsequence such that $\chi_{\{z_{\varepsilon_k}=0\}}\stackrel{\star}{\rightharpoonup}\chi$
				in $L^\infty(\Omega_T)$ as $k\rightarrow\infty$. In the following, we will omit the index $k$ in the notation.
				Integrating \eqref{eqn:weaksolution4} from $t=0$ to $t=T$ and passing to $\varepsilon\searrow 0$ show
				\begin{align}
					&\int_{\Omega_T}\nabla z\cdot\nabla\zeta
						+(W_{,z}^\mathrm{el}(e(u),c,z)
						-\alpha+\beta(\partial_t z))\zeta\,\mathrm dx
						\geq\int_{\Omega_T} \chi[W_{,z}^\mathrm{el}(e(u),c,z)]^+\zeta\,\mathrm dx\mathrm ds
				\label{eqn:limitVI}
				\end{align}
				for all $\zeta\in L^p(0,T;W_-^{1,p}(\Omega))\cap L^\infty(\Omega_T)$.
				We also used the fact that
				\begin{align*}
					\left|\int_{\Omega_T}\varepsilon|\nabla z_\varepsilon|^{p-2}\nabla z_\varepsilon\cdot\nabla\zeta\,\mathrm dx\mathrm ds\right|
					\leq \varepsilon \|\nabla z_\varepsilon\|_{L^p(\Omega_T)}^{p-1}\|\nabla\zeta\|_{L^p(\Omega_T)}\rightarrow 0.
				\end{align*}
				It follows that
				\begin{align*}
					&\int_{\Omega}\nabla z(t)\cdot\nabla\zeta
						+(W_{,z}^\mathrm{el}(e(u(t)),c(t),z(t))
						-\alpha+\beta(\partial_t z(t)))\zeta\,\mathrm dx\\
					&\qquad\qquad\geq\int_{\Omega} \chi(t)[W_{,z}^\mathrm{el}(e(u(t)),c(t),z(t))]^+\zeta\,\mathrm dx
				\end{align*}
				for all $\zeta\in H_-^1(\Omega)\cap L^\infty(\Omega)$ and a.e. $t\in[0,T]$.
				Set $r:=-\chi[W_{,z}^\mathrm{el}(e(u),c,z)]^+$.
				For every $\xi\in L^\infty([0,T])$ with $\xi\geq 0$ a.e. on $[0,T]$ and every
				$\zeta\in H_+^1(\Omega)\cap L^\infty(\Omega)$ we also have
				\begin{align*}
						0\geq{}&\int_0^T\left(\int_{\Omega} r_\varepsilon(t)(\zeta-z_\varepsilon(t))\,\mathrm dx\right)\xi(t)\,\mathrm dt
							=\int_{\Omega_T} r_\varepsilon(\zeta-z_\varepsilon)\xi\,\mathrm dx\mathrm dt\\
						&\rightarrow
							\int_{\Omega_T} r(\zeta-z)\xi\,\mathrm dx\mathrm dt
							=\int_0^T\left(\int_{\Omega} r(t)(\zeta-z(t))\,\mathrm dx\right)\xi(t)\,\mathrm dt.
				\end{align*}
				This shows $\int_{\Omega} r(t)(\zeta-z(t))\,\mathrm dx\leq 0$ for a.e. $t\in[0,T]$.
				Hence, we obtain the inequalities (v) of Definition \ref{def:weakSolutionLimit}.
			\item
				Weakly semi-continuity arguments lead to
				\begin{align*}
					&\liminf_{\varepsilon\searrow 0}\Big(\mathcal E_\varepsilon(u_\varepsilon(t),c_\varepsilon(t),z_\varepsilon(t))
						+\int_{\Omega_t} \alpha|\partial_t z_\varepsilon|+\beta |\partial_t z_\varepsilon|^2
						+\varepsilon|\partial_t c_\varepsilon|^2\,\mathrm dx\mathrm ds
						+\int_0^t\langle\mathcal Sw_\varepsilon,w_\varepsilon\rangle\,\mathrm ds\Big)\\
					&\qquad\qquad\geq\mathcal E(u(t),c(t),z(t))
						+\int_{\Omega_t} \alpha|\partial_t z|+\beta |\partial_t z|^2
						+\int_0^t\langle\mathcal Sw,w\rangle\,\mathrm ds.
				\end{align*}
				
				Testing \eqref{eqn:weaksolution3} with $\zeta=u_\varepsilon^0-b(0)$
				and (iv) of Definition \ref{def:weakSolutionLimit} with $\zeta=u^0-b(0)$ yield
				\begin{align*}
					\varepsilon\int_{\Omega}|\nabla u_\varepsilon^0|^4\,\mathrm dx
						={}&\varepsilon\int_{\Omega}|\nabla u_\varepsilon^0|^2\nabla u_\varepsilon^0:\nabla b(0)\,\mathrm dx\\
					&-\int_{\Omega}W_{,e}^\mathrm{el}(e(u_\varepsilon^0),c^0,z_\varepsilon^0):e(u_\varepsilon^0-b(0))
						\,\mathrm dx\\
					\rightarrow{}&
						-\int_{\Omega}W_{,e}^\mathrm{el}(e(u^0),c^0,z^0):e(u^0-b(0))\,\mathrm dx=0
				\end{align*}
				as $\varepsilon\searrow 0$.
				
				Therefore, we can pass to the limit $\varepsilon\searrow 0$ in \eqref{eqn:weaksolution5} and obtain
				(vi) from Definition \ref{def:weakSolutionLimit}.\ep
		\end{enumerate}
	\end{proof}
	
\section{Higher integrability of the strain tensor}
	To prove existence results for chemical free energies of logarithmic type, a higher integrability result
	for the strain tensor based on \cite{GarckeHabil, Garcke05} will be established.
	We adapt the higher integrability result for solutions of the elliptic equation of the form
	\begin{align*}
		\left.
		\begin{cases}
			\mathrm{div}(W_{,e}^\mathrm{el}(e(u),c))=0&\text{ on }\Omega_T,\\
			W_{,e}^\mathrm{el}(e(u),c)\cdot\overrightarrow{\nu} =\sigma^\star\cdot\overrightarrow{\nu}&\text{ on }(\partial\Omega)_T
		\end{cases}
		\right\}
	\end{align*}
	to our setting with non-constant Dirichlet boundary data $b$ and the additional damage variable $z$
	in \eqref{eqn:unifyingModelClassical}.
	In the following, we will use the assumption $D=\partial\Omega$.
	
	The proof of the higher integrability result is based on
	the following special cases of the Sobolev-Poincar\'e inequalities
	and on a reverse H\"older inequality.
	\begin{theorem}[Sobolev-Poincar\'e type inequalities]
	\label{theorem:SobPoincare}
		Let $1\leq p<n$. There exists a constant $C>0$ such that
		\begin{enumerate}
		 \item[(i)]
			for all rectangles $Q\subseteq\mathbb R^n$ and all $u\in W^{1,p}(Q)$:
				\begin{align*}
					\left(\dashint_Q|u-\dashint_Q u|^{p^\star}\right)^{\frac{1}{p^\star}}
						\leq C\left(\dashint_Q|\nabla u|^p\right)^{\frac 1p}(\mathrm{diam}Q),
				\end{align*}
			\item[(ii)]
				for all rectangles $Q=\prod_{i=1}^n (a_i,b_i)\subseteq\mathbb R^n$ and all $u\in W^{1,p}(Q)$ with
				$u=0$ on\linebreak $\big\{(x_1,\ldots, x_{n-1},a_n)\,|\,a_i\leq x_i\leq b_i,\;i=1,\ldots, n-1\big\}\subseteq \partial Q$
				(in the sense of traces):
				\begin{align*}
					\left(\dashint_Q|u|^{p^\star}\right)^{\frac{1}{p^\star}}
						\leq C\left(\dashint_Q|\nabla u|^p\right)^{\frac 1p}(\mathrm{diam}Q).
				\end{align*}
		\end{enumerate}
	\end{theorem}
	Theorem \ref{theorem:SobPoincare} can be obtained by considering the corresponding inequalities on the unit cube $(0,1)^n$
	(for instance the case $1<p<n$ was proven by Sobolev \cite{Sob36} while Nirenberg \cite{Nirenberg59} gave a proof to $p=1$)
	and then using a scaling argument.
	\begin{theorem}[Reverse H\"older inequality, see \cite{Giaquinta83}]
	\label{theorem:revHoelder}
		Let $Q\subseteq\mathbb R^n$ be a cube, $g\in L_\mathrm{loc}^q(Q)$ for some $q>1$ and $g\geq 0$.
		Suppose that there exist a constant $b>0$ and a function $f\in L_\mathrm{loc}^r(Q)$ with $r>q$ and $f\geq 0$
		such that
		\begin{align*}
			\dashint_{Q_R(x_0)}g^q\,\mathrm dx\leq b\left(\dashint_{Q_{2R}(x_0)}g\,\mathrm dx\right)^q
			+\dashint_{Q_{2R}(x_0)}f^q\,\mathrm dx
		\end{align*}
		for each $x_0\in Q$ and all $R>0$ with $2R<\mathrm{dist}(x_0,\partial Q)$.
		Then $g\in L_\mathrm{loc}^s(Q)$ for $s\in[q,q+\varepsilon)$ with some $\varepsilon>0$ and
		\begin{align*}
			\left(\dashint_{Q_R(x_0)}g^s\,\mathrm dx\right)^{\frac 1s}\leq c\left(\left(\dashint_{Q_{2R}(x_0)}g^q
			\,\mathrm dx\right)^{\frac 1q}+\left(\dashint_{Q_{2R}(x_0)}f^s\,\mathrm dx\right)^{\frac 1s}\right)
		\end{align*}
		for all $x_0\in Q$ and $R>0$ such that $Q_{2R}(x_0)\subseteq Q$. The positive constants $c,\varepsilon>0$ depend
		on $b$, $q$, $n$ and $r$.
	\end{theorem}
	\begin{theorem}[Higher integrability]
	\label{theorem:integrability}
		Let $b\in W^{1,\infty}(\Omega;\mathbb R^n)$, $z\in L^\infty(\Omega)$ with $0\leq z\leq 1$ a.e.
		in $\Omega$ and $c\in L^\mu(\Omega;\mathbb R^N)$ for some $\mu>4$.
		Then there exists some $p\in(2,\mu/2]$ such that for all $u\in H^1(\Omega;\mathbb R^n)$
		which satisfy $u|_D=b|_D$ and
		\begin{align}
		\label{eqn:higherItegrabilityCond}
			\int_\Omega W_{,e}^\mathrm{el}(e(u),c,z):e(\zeta)\,\mathrm dx=0\text{ for all }\zeta\in H_D^1(\Omega;\mathbb R^n),
		\end{align}
		we obtain $u\in W^{1,p}(\Omega;\mathbb R^n)$ and
		\begin{align}
		\label{eqn:higherIntegrabilityEstimate}
			\|\nabla u\|_{L^p(\Omega;\mathbb R^{n\times n})}\leq C(\|\nabla u\|_{L^2(\Omega;\mathbb R^{n\times n})}
				+\|c\|_{L^{2p}(\Omega;\mathbb R^N)}^2+1).
		\end{align}
		The positive constants $p$ and $C$ are independent of $u$, $c$, $z$.
	\end{theorem}
	\begin{proof}
		The proof is based on \cite[Lemma 4.4 and Theorem 4.3]{GarckeHabil} and uses a covering argument.
		However, due to the non-constant boundary condition, we need to apply a more general Sobolev-Poincar\'e inequality
		(see Theorem \ref{theorem:SobPoincare} (ii)) than in \cite{GarckeHabil}.
		\begin{enumerate}
			\renewcommand{\labelenumi}{(\roman{enumi})}
			\item\textsc{Higher integrability at the boundary. }
			
				Let $x_0\in\partial\Omega$. Then there exist an $R_0>0$ and a bi-Lipschitz function $\tau:Q\rightarrow\mathbb R^n$
				with the open cube $Q:=Q_{R_0}(0)$ such that $x_0\in \tau(Q)$ and
				\begin{align*}
					&\tau(Q^+)\subseteq\Omega,\\
					&\tau(Q^-)\subseteq\mathbb R^n\setminus\overline{\Omega},
				\end{align*}
				where $Q^+:=\{x\in Q\,|\,x_n>0\}$ and $Q^-:=\{x\in Q\,|\,x_n<0\}$.
				Define the transformed functions $\tilde u,\tilde b\in H^1(Q^+;\mathbb R^n)$, $\tilde c\in H^1(Q^+)$ and
				$\tilde z\in L^\infty(Q^+)$ as
				\begin{align*}
					(\tilde u,\tilde b,\tilde c,\tilde z)(x):=
						(u,b,c,z)(\tau(x)).
				\end{align*}
				To proceed, let $y_0\in Q$ and $R<\frac 12\mathrm{dist}(y_0,\partial Q)$ and 
				define for each $R'>0$ the sets
				\begin{align*}
					Q_{R'}^{\pm}(y_0):=\{x\in Q_{R'}(y_0)\,|\,x_n\gtrless 0\}.
				\end{align*}
				
				We distinguish three cases:
				
				\textbf{Case 1. }We first consider the case $Q_R^+(y_0)\neq\emptyset$ and $Q_{\frac32 R}^-(y_0)\neq\emptyset$.
				
				The bi-Lipschitz continuity of $\tau$ ensures
				\begin{align*}
					\mathrm{dist}(\tau(\partial Q_{2R}^+(y_0))\cap \Omega,\tau(\partial Q_R^+(y_0))\cap \Omega)>RC_1,
				\end{align*}
				where $C_1>0$ is independent of $R$ and $y_0$.
				Let $\xi\in \C C^\infty_0(\Omega)$ be a cutoff function with the properties:
				
				\begin{tabular}{ll}
					\begin{minipage}{19em}
						\begin{enumerate}
							\renewcommand{\labelenumi}{(\roman{enumi})}
							\item[(a)] $\xi=0\text{ in }\Omega\setminus\tau(Q_{2R}(y_0))$,
							\item[(b)] $0\leq\xi\leq 1\text{ in }\Omega$,
						\end{enumerate}
					\end{minipage}
					&
					\begin{minipage}{21em}
						\begin{enumerate}
							\renewcommand{\labelenumi}{(\roman{enumi})}
							\item[(c)] $\xi\equiv 1\text{ in }\tau(Q_R(y_0))\cap\Omega$,
							\item[(d)] $|\nabla\xi|\leq \frac{2}{C_1}R^{-1}$.
						\end{enumerate}
					\end{minipage}
				\end{tabular}
				
				Testing \eqref{eqn:higherItegrabilityCond} with $\zeta=\xi^2(u-b)$, using the computation
				\begin{align*}
					e(\zeta)=\xi^2 e(u)-\xi^2 e(b)+\xi((u-b)(\nabla\xi)^t+\nabla\xi(u-b)^t),
				\end{align*}
				and \eqref{eqn:growthEst1}, we obtain
				\begin{align}
					&\int_\Omega\xi^2 W_{,e}^\mathrm{el}(e(u),c,z):e(u)\,\mathrm dx
						\notag\\
					&\quad=\int_\Omega\xi^2 W_{,e}^\mathrm{el}(e(u),c,z):e(b)\,\mathrm dx
					-2\int_\Omega\xi W_{,e}^\mathrm{el}(e(u),c,z):((u-b)(\nabla\xi)^t)\,\mathrm dx.
					\label{eqn:temp1}
				\end{align}
				By \eqref{eqn:growthEst3}, \eqref{eqn:growthEst4} and \eqref{eqn:growthEst2} we also have the estimates
				\begin{align*}
					&\eta|e(u)|^2\leq W_{,e}^\mathrm{el}(e(u),c,z):e(u)+C(|c|^2+1)|e(u)|,\\
					&|W_{,e}^\mathrm{el}(e(u),c,z):((u-b)(\nabla\xi)^t|\leq \frac{C}{R}(|e(u)|+|c|^2+1)|u-b|,\\
					&|W_{,e}^\mathrm{el}(e(u),c,z):e(b)|\leq (|e(u)|+|c|^2+1)|e(b)|.
				\end{align*}
				Therefore, \eqref{eqn:temp1} can be estimated by
				\begin{align*}
					\eta\int_\Omega\xi^2|e(u)|^2\,\mathrm dx
						\leq{}& C\int_\Omega\xi^2(|c|^2+1)|e(u)|\,\mathrm dx
						+\frac{C}{R}\int_\Omega\xi(|e(u)|+|c|^2+1)|u-b|\,\mathrm dx\\
					&+C\int_\Omega\xi^2(|e(u)|+|c|^2+1)|e(b)|\,\mathrm dx.
				\end{align*}
				Young's inequality yields
				\begin{align}
				\label{eqn:regProofEst1}
					c_1\int_\Omega\xi^2|e(u)|^2\,\mathrm dx
						\leq{}& C\int_{\Omega}\xi^2(|c|^4+1)\,\mathrm dx
						+\frac{C}{R^2}\int_{\Omega}|u-b|^2\,\mathrm dx.
				\end{align}
				We choose $\mu=\dashint_{Q_{2R}^+(y_0)} \tilde u\,\mathrm dx$.
				The calculation $e(\xi(u-\mu))=\xi e(u)+\frac12((u-\mu)(\nabla\xi)^t+\nabla\xi(u-\mu)^t)$ leads to
				\begin{align}
				\label{eqn:regProofEst2}
					\int_\Omega|e(\xi(u-\mu))|^2\,\mathrm dx\leq 2\left(\int_\Omega\xi^2|e(u)|^2\,\mathrm dx+\int_\Omega|u-\mu|^2|\nabla\xi|^2
						\,\mathrm dx\right).
				\end{align}
				Combining \eqref{eqn:regProofEst1} and \eqref{eqn:regProofEst2}, applying Korn's inequality for $H^1$-functions with zero boundary values and using (a) and (b)
				gives
				\begin{align*}
					\int_\Omega|\nabla(\xi(u-\mu))|^2\,\mathrm dx
					\leq{}&C\int_{\tau(Q_{2R}^+(y_0))}(|c|^4+1)\,\mathrm dx
						+\frac{C}{R^2}\int_{\tau(Q_{2R}^+(y_0))}|u-b|^2\,\mathrm dx\\
					&+\frac{C}{R^2}\int_{\tau(Q_{2R}^+(y_0))}|u-\mu|^2\,\mathrm dx.
				\end{align*}
				Because of $\nabla(\xi(u-\mu))=\xi\nabla u+(u-\mu)(\nabla\xi)^t$ we derive by (a) and (c)
				the following type of Caccioppoli-inequality:
				\begin{align*}
					\int_{\tau(Q_R^+(y_0))}|\nabla u|^2\,\mathrm dx
					\leq{}&C\int_{\tau(Q_{2R}^+(y_0))}(|c|^4+1)\,\mathrm dx
						+\frac{C}{R^2}\int_{\tau(Q_{2R}^+(y_0))}|u-b|^2\,\mathrm dx\\
					&+\frac{C}{R^2}\int_{\tau(Q_{2R}^+(y_0))}|u-\mu|^2\,\mathrm dx.
				\end{align*}
				Integral transformation by $\tau$ implies
				\begin{align*}
					\int_{Q_R^+(y_0)}|\nabla \tilde u|^2\,\mathrm dx
					\leq{}&C\int_{Q_{2R}^+(y_0)}(|\tilde c|^4+1)\,\mathrm dx
						+\frac{C}{R^2}\int_{Q_{2R}^+(y_0)}|\tilde u-\tilde b|^2\,\mathrm dx\\
					&+\frac{C}{R^2}\int_{Q_{2R}^+(y_0)}|\tilde u-\mu|^2\,\mathrm dx.
				\end{align*}
				The condition $Q_{\frac32 R}^-(y_0)\neq\emptyset$ and $D=\partial\Omega$ imply
				that $\tilde u-\tilde b$ vanishes on $\partial\big(Q_{2R}^+(y_0)\big)\cap\mathbb R^{n-1}\times\{0\}$.
				Therefore, we obtain by applying both variants of the Poincar\'e-Sobolev inequality in Theorem \ref{theorem:SobPoincare}
				for $p=2n/(n+2)$:
				\begin{align}
					\int_{Q_R^+(y_0)}|\nabla \tilde u|^2\,\mathrm dx
					\leq{}&C\int_{Q_{2R}^+(y_0)}(|\tilde c|^4+1)\,\mathrm dx
						+\frac{C}{R^2}\mathcal L^n(Q_{2R}^+(y_0))^{-\frac{2}{n}}\mathrm{diam}(Q_{2R}^+(y_0))^2\notag\\
						&\cdot\left[\left(\int_{Q_{2R}^+(y_0)}|\nabla \tilde u-\nabla \tilde b|^\frac{2n}{n+2}
						\,\mathrm dx\right)^\frac{n+2}{n}
						+\left(\int_{Q_{2R}^+(y_0)}|\nabla \tilde u|^\frac{2n}{n+2}\,\mathrm dx\right)^\frac{n+2}{n}\right].
					\label{eqn:temp2}
				\end{align}
				Note that if $n=1$ we cannot apply Theorem \ref{theorem:SobPoincare} because of $p=2n/(n+2)<1$.
				In this case, we can work with the inequalities in Theorem \ref{theorem:SobPoincare} where $p$ is substituted by $1$ and
				$p^\star$ is substituted by $2$. However, we will only treat the more delicate case $n\geq 2$ in the following.
				
				The estimates $\mathrm{diam}(Q_{2R}^+(y_0))\leq CR$ and $\mathcal L^n(Q_{2R}^+(y_0))\geq R^n$ (because of
				$Q_R^+(y_0)\neq\emptyset$) show
				\begin{align}
				\label{eqn:LnEstimate}
					\mathcal L^n(Q_{2R}^+(y_0))^{-\frac{2}{n}}\mathrm{diam}(Q_{2R}^+(y_0))^2\leq C.
				\end{align}
				Now, dividing \eqref{eqn:temp2} by $\mathcal L^n(Q_R(y_0))$ and using \eqref{eqn:LnEstimate} and
				\begin{align*}
					\frac{1}{R^2}\frac{1}{\mathcal L^n(Q_{2R}(y_0))}\leq C\left(\frac{1}{\mathcal L^n(Q_{2R}(y_0))}\right)^\frac{n+2}{n}
				\end{align*}
				gives
				\begin{align*}
					\frac{1}{\mathcal L^n(Q_R(y_0))}\int_{Q_R^+(y_0)}|\nabla \tilde u|^2\,\mathrm dx
					\leq{}&\frac{C}{\mathcal L^n(Q_{2R}(y_0))}\int_{Q_{2R}^+(y_0)}(|\tilde c|^4+1)\,\mathrm dx\\
						&+C\left(\frac{1}{\mathcal L^n(Q_{2R}(y_0))}
						\int_{Q_{2R}^+(y_0)}|\nabla \tilde u|^\frac{2n}{n+2}\,\mathrm dx\right)^\frac{n+2}{n}\\
						&+C\left(\frac{1}{\mathcal L^n(Q_{2R}(y_0))}
						\int_{Q_{2R}^+(y_0)}|\nabla \tilde b|^\frac{2n}{n+2}\,\mathrm dx\right)^\frac{n+2}{n}.
				\end{align*}
				Observe that
				\begin{align*}
					\left(\frac{1}{\mathcal L^n(Q_{2R}(y_0))}
						\int_{Q_{2R}^+(y_0)}|\nabla \tilde b|^\frac{2n}{n+2}\,\mathrm dx\right)^\frac{n+2}{n}
						\leq \|\nabla b\|_{L^\infty(\Omega)}^2.
				\end{align*}
				Define the following functions on $Q$:
				\begin{align*}
					g(x):=
					\begin{cases}
						|\nabla\tilde u(x)|^{\frac{2n}{n+2}}&\text{ for }x\in Q^+,\\
						0&\text{ for }x\in Q\setminus Q^+
					\end{cases}
				\end{align*}
				and
				\begin{align*}
					f(x):=
					\begin{cases}
						C(|\tilde c|^4+\|\nabla b\|_{L^\infty(\Omega)}^2+1)^\frac{n}{n+2}&\text{ for }x\in Q^+,\\
						0&\text{ for }x\in Q\setminus Q^+.
					\end{cases}
				\end{align*}
				We eventually get
				\begin{align}
				\label{eqn:regProofEst3}
					\dashint_{Q_R(y_0)}g^{\frac{n+2}{n}}\,\mathrm dx
					\leq\dashint_{Q_{2R}(y_0)}f^\frac{n+2}{n}\,\mathrm dx
						+C\left(\dashint_{Q_{2R}(y_0)}g\,\mathrm dx\right)^\frac{n+2}{n}.
				\end{align}
				
				\textbf{Case 2. } Assume $Q_R^+(y_0)\neq\emptyset$ and $Q_{\frac32 R}^-(y_0)=\emptyset$.
				
				The bi-Lipschitz continuity of $\tau$ implies
				\begin{align*}
					\mathrm{dist}(\tau(\partial Q_{\frac 32R}(y_0)),\tau(\partial Q_R(y_0)))>RC_1,
				\end{align*}
				where $C_1>0$ is independent of $R$ and $y_0$.
				Therefore, we can choose a cutoff function $\xi\in \C C^\infty_0(\Omega)$ which satisfies
				
				\begin{tabular}{ll}
					\begin{minipage}{19em}
						\begin{enumerate}
							\renewcommand{\labelenumi}{(\roman{enumi})}
							\item[(a)] $\xi=0\text{ in }\Omega\setminus \tau(Q_{\frac 32R}(x_0))$,
							\item[(b)] $0\leq\xi\leq 1\text{ in }\Omega$,
						\end{enumerate}
					\end{minipage}
					&
					\begin{minipage}{21em}
						\begin{enumerate}
							\renewcommand{\labelenumi}{(\roman{enumi})}
							\item[(c)] $\xi\equiv 1\text{ in } \tau(Q_R(x_0))$,
							\item[(d)] $|\nabla\xi|\leq \frac{2}{C_1}R^{-1}$.
						\end{enumerate}
					\end{minipage}
				\end{tabular}
				
				Testing \eqref{eqn:higherItegrabilityCond} with $\xi=\zeta^2(u-\mu)$
				and $\mu:=\dashint_{Q_{\frac32R}(x_0)}\tilde u\,\mathrm dx$ yields as in the previous case
				\begin{align*}
					\int_{\tau(Q_R(x_0))}|\nabla u|^2\,\mathrm dx
					\leq{}&C\int_{\tau(Q_{\frac32R}(x_0))}(|c|^4+1)\,\mathrm dx
						+\frac{C}{R^2}\int_{\tau(Q_{\frac32R}(x_0))}|u-\mu|^2\,\mathrm dx.
				\end{align*}
				Consequently,
				\begin{align*}
					\dashint_{Q_R(x_0)}|\nabla \tilde u|^2\,\mathrm dx
					\leq{}&C\dashint_{Q_{\frac32R}(x_0)}(|\tilde c|^4+1)\,\mathrm dx
						+C\left(\dashint_{Q_{\frac32R}(x_0)}|\nabla \tilde u|^\frac{2n}{n+2}\,\mathrm dx\right)^\frac{n+2}{n}.
				\end{align*}
				Therefore, the inequality \eqref{eqn:regProofEst3} is also satisfied in this case.
				
				\textbf{Case 3. } Assume $Q_R^+(y_0)=\emptyset$.
				
				In this case, inequality \eqref{eqn:regProofEst3} trivially holds.

				In all three cases, the reverse H\"older inequality (see Theorem \ref{theorem:revHoelder}) shows $g\in L_\mathrm{loc}^s(Q)$ for all
				$s\in\big[\frac{n+2}{n},\frac{n+2}{n}+\varepsilon\big)$ and some $\varepsilon>0$ depending on
				$R_0$ and $n$.

			\item\textsc{Higher integrability in the interior. }
			
				This case follows with much less effort and is only sketched here.
				
				Let $x_0\in\Omega$ be arbitrary and $R>0$ such that
				$Q_{2R}(x_0)\subseteq\Omega$.
				We take a cutoff function	$\xi\in \C C^\infty_0(\Omega)$ with
				
				\begin{tabular}{ll}
					\begin{minipage}{19em}
						\begin{enumerate}
							\renewcommand{\labelenumi}{(\roman{enumi})}
							\item[(a)] $\xi=0\text{ in }\Omega\setminus Q_{2R}(x_0)$,
							\item[(b)] $0\leq\xi\leq 1\text{ in }\Omega$,
						\end{enumerate}
					\end{minipage}
					&
					\begin{minipage}{21em}
						\begin{enumerate}
							\renewcommand{\labelenumi}{(\roman{enumi})}
							\item[(c)] $\xi\equiv 1\text{ in } Q_R(x_0)$,
							\item[(d)] $|\nabla\xi|\leq \frac{2}{R}$.
						\end{enumerate}
					\end{minipage}
				\end{tabular}
				
				Testing \eqref{eqn:higherItegrabilityCond} with $\xi=\zeta^2(u-\mu)$
				and $\mu=\dashint_{Q_{2R}(x_0)}u\,\mathrm dx$ yields with the same computation as in the case (i):
				\begin{align*}
					\int_{Q_R(x_0)}|\nabla u|^2\,\mathrm dx
					\leq{}&C\int_{Q_{2R}(x_0)}(|c|^4+1)\,\mathrm dx
						+\frac{C}{R^2}\int_{Q_{2R}(x_0)}|u-\mu|^2\,\mathrm dx.
				\end{align*}
				The Poincar\'e-Sobolev inequality implies
				\begin{align*}
					\dashint_{Q_R(x_0)}|\nabla u|^2\,\mathrm dx
					\leq{}&C\dashint_{Q_{2R}(x_0)}(|c|^4+1)\,\mathrm dx
						+C\left(\dashint_{Q_{2R}(x_0)}|\nabla u|^\frac{2n}{n+2}\,\mathrm dx\right)^\frac{n+2}{n}.
				\end{align*}
				Applying Theorem \ref{theorem:revHoelder} with $g=|\nabla u|^\frac{2n}{n+2}$, $q=\frac{n+2}{n}$
				and $f=C(|c|^4+1)^\frac{n}{n+2}$ finishes the proof.
				\ep
		\end{enumerate}
	\end{proof}

\section{Existence of weak solutions of \eqref{eqn:unifyingModelClassical} - logarithmic case}

	The challenge here is to establish the integral equation (iii) in Definition \ref{def:weakSolutionLimit} because
	the derivative of the logarithmic free chemical energy
        \eqref{eqn:growthEst8} becomes singular if one of the $c_k$'s approaches $0$.
	We only sketch the proof in this section since all essential ideas can be found in \cite{GarckeHabil, Garcke05}.
	We use a regularization method suggested in \cite{EL91} and also used
        in \cite{GarckeHabil, Garcke05}.

	The energy gradient tensor is assumed to be of the form
        $\mathbf\Gamma=\gamma \, \mathrm{Id}$ with a constant $\gamma>0$. 
	Define a $\C C^2(\mathbb R^N)$ regularization with the regularization parameter $\delta>0$ as
	\begin{align*}
		W^{\mathrm{ch},\delta}(c):=\theta\sum_{k=1}^N\phi^\delta(c^k)+\frac 12c\cdot Ac,
	\end{align*}
	with
	\begin{align*}
		\phi^\delta(x):=
		\begin{cases}
			x\log(x)&\text{for }d\geq\delta,\\
			x\log(\delta)-\frac\delta2+\frac{x^2}{2\delta}&\text{for }x<\delta.
		\end{cases}
	\end{align*}
	Elliott and Luckhaus showed that the regularization $W^{\mathrm{ch},\delta}$ is uniformly bounded from below.
	\begin{lemma}[cf.~ \cite{EL91}]
	\label{lemma:uniformWchEstimate}
		There exist constants $\delta_0>0$ and $C>0$ such that
		\begin{align*}
			W^{\mathrm{ch},\delta}(c)\geq -C\qquad\text{for all } c\in\Sigma,\;\delta\in(0,\delta_0).
		\end{align*}
	\end{lemma}
	
	Let $q_\delta$ denote a weak solution in the sense of Definition \ref{def:weakSolutionLimit} 
	with the free chemical energy $W^\mathrm{ch}=W^{\mathrm{ch},\delta}$.
	By applying Lemma \ref{lemma:uniformWchEstimate} and using Gronwall's inequality in the energy inequality (vi) of Definition
	\ref{def:weakSolutionLimit}, we can show a-priori estimates
	analogous as in Section 4 except the a-priori estimate of $w_\delta$.
	
	In the Allen-Cahn case, we have $\partial_t c_\delta=-\mathbb M w_\delta$ and, consequently, the boundedness of $c_\delta$ in $L^2(\Omega;\mathbb R^N)$
	and $w_\delta\in T\Sigma$ pointwise lead to boundedness of $w_\delta$ in $L^2(\Omega;\mathbb R^N)$.
	
	In the case of Cahn-Hilliard systems, we can use the following lemma.
	\begin{lemma}[{\cite[Lemma 4.3]{GarckeHabil}}] \label{le:log}
		There exists a constant $C>0$ such that for all $\delta\in(0,\delta_0)$
		\begin{align*}
			\int_0^T\left(\dashint_\Omega\mathbb P W_{,c}^{\mathrm{ch},\delta}(c_\delta(t))\,\mathrm dx\right)^2\mathrm dt<C.
		\end{align*}
	\end{lemma}
	 The proof of  this lemma is similar to \cite[Lemma
           4.3]{GarckeHabil}, since all arguments can be adapted to our
         case. Therefore we will omit the proof. 
	
	This lemma and the integral equation
	\begin{equation*}
		\begin{split}
			\int_{\Omega} w_\delta(t)\,\mathrm dx
				=&\int_{\Omega}\mathbb P W_{,c}^{\mathrm{ch},\delta}(c_\delta(t))+
				\mathbb P W_{,c}^\mathrm{el}(e(u_\delta(t)),c_\delta(t),z_\delta(t))\,\mathrm dx
		\end{split}
	\end{equation*}
	together with the already known boundedness properties shows
	\begin{align*}
		\int_0^T\left(\dashint_\Omega w_\delta(t)\,\mathrm dx\right)^2\mathrm dt<C
	\end{align*}
	for a constant $C>0$. Therefore $w_\delta$ is bounded in $L^2(0,T;H^1(\Omega))$ by Poincar\'e's inequality.
	In conclusion, we can extract a subsequence $\{q_{\delta_k}\}$ such that we have the same convergence properties as in
	Lemma \ref{lemma:convergenceProperties2}. As before, we will omit the subscript $k$.

	\begin{proof}[Proof of Theorem \ref{theorem:mainTheoremLog}]
	The remaining crucial step is to show that the limit $c$ satisfies $c_k>0$ a.e. on $\Omega_T$ for all $k=1,\ldots,N$
	and $W_{,c}^{\mathrm{ch},\delta}(c_\delta)\rightarrow W_{,c}^\mathrm{ch,log}(c)$ in $L^1(\Omega_T)$ as $\varepsilon\searrow 0$.
	
	To this end, we need an additional boundedness property.
	\begin{lemma}
	\label{lemma:phiDeltaBoundedness}
		There exists constants $q>1$ and $C>0$ such that for all $\delta\in(0,\delta_0)$ and all $k=1,\ldots,N$
		\begin{align*}
			\|(\phi^\delta)'(c_\delta^k)\|_{L^q(\Omega_T)}< C.
		\end{align*}
	\end{lemma}
	We omit the proof of this lemma, since by utilizing Theorem \ref{theorem:integrability} the  arguments are analogous to
	\cite[Lemma 4.5]{GarckeHabil}.

		Note that
		\begin{align*}
			\lim_{\delta\searrow 0}(\phi^\delta)'(c_\delta^k)=
			\begin{cases}
				\log(c^k)+1&\text{if }\lim_{\delta\searrow0}c_\delta^k=c^k>0,\\
				\infty&\text{otherwise}
			\end{cases}
		\end{align*}
		holds pointwise a.e. on $\Omega_T$ and for all $k=1,\ldots,N$.
		Together with Lemma \ref{lemma:phiDeltaBoundedness}, we obtain
		\begin{align*}
			c^k>0\text{ a.e. on }\Omega_T
		\end{align*}
		and
		\begin{align*}
			(\phi^\delta)'(c_\delta^k)\rightarrow \log(c^k)+1\text{ a.e. on }\Omega_T.
		\end{align*}
		This and Lemma \ref{lemma:phiDeltaBoundedness} further shows
		\begin{align*}
			(\phi^\delta)'(c_\delta^k)\rightarrow \log(c^k)+1\text{ in } L^1(\Omega_T)
		\end{align*}
		by Vitali's convergence theorem.
		Finally, we can pass to $\delta\searrow 0$ in the equation 
		\begin{align*}
			\int_{\Omega_T} w_\delta\cdot\zeta\,\mathrm dx\mathrm dt
				=&\int_{\Omega_T}\gamma\nabla c_\delta:\nabla\zeta
				+\mathbb P W_{,c}^{\mathrm{ch},\delta}(c_\delta)\cdot\zeta
				+\mathbb P W_{,c}^\mathrm{el}(e(u_\delta),c_\delta,z_\delta)\cdot\zeta\,\mathrm dx\mathrm dt
		\end{align*}
		and obtain (iii) from Definition \ref{def:weakSolutionLimit}.
		
		The remaining properties can be easily established as in Section 4.
		Hence, Theorem \ref{theorem:mainTheoremLog} is proven.\ep
	\end{proof}

\section{Conclusion}
Materials, which enable the functionality of technical products,
change the micro-structure over time. Phase separation and coarsening
phenomena take place and the complete failure of electronic devices
often results from micro-cracks in solder joints.

In this work, we have investigated mathematical
models describing both phenomena, phase separation and damage
processes, in a unifying approach. The main aim has been to prove
existence of weak solutions for elastic Cahn-Hilliard and Allen-Cahn
systems coupled with damage phenomena under mild assumptions
where the free energy contains
\begin{itemize}
	\item
		a chemical potential of polynomial or logarithmic type,
	\item
		an inhomogeneous elastic energy, e.g. $W^\mathrm{el}(e,c,z)=\frac 12(z+\varepsilon)\mathbb C(c)(e-e^\star(c)):(e-e^\star(c))$,
	\item
		a quadratic gradient term of the damage variable.
\end{itemize}
To this end, several approxmation results have been established
as well as different variational techniques, regularization methods and higher integrability results for the strain
have been applied.

	
\addcontentsline{toc}{chapter}{Bibliography}{\footnotesize{\setlength{\baselineskip}{0.2 \baselineskip}
\bibliography{references}}
\bibliographystyle{alpha}}
\end{document}